\documentclass[12pt]{amsart}
\usepackage{fullpage,verbatim,amssymb}
\usepackage{hyperref}
\makeatletter
\let\@@pmod\pmod
\DeclareRobustCommand{\pmod}{\@ifstar\@pmods\@@pmod}
\def\@pmods#1{\mkern4mu({\operator@font mod}\mkern 6mu#1)}
\makeatother

\newcommand{\Z}{\mathbb{Z}}
\newcommand{\Q}{\mathbb{Q}}
\newcommand{\R}{\mathbb{R}}
\newcommand{\C}{\mathbb{C}}
\newcommand{\HH}{\mathbb{H}}
\newcommand{\DD}{\mathcal{D}}
\newcommand{\hh}{\mathbf{h}^+}
\newcommand{\A}{\mathcal{A}}
\newcommand{\Anew}{\mathcal{A}^\mathrm{new}}
\newcommand{\rr}{\mathfrak{r}}
\newcommand{\Beta}{\mathrm{B}}
\newcommand{\IN}{\mathcal{I}_N}

\DeclareMathOperator{\Sym}{Sym}
\DeclareMathOperator{\ord}{ord}
\DeclareMathOperator{\Tr}{Tr}
\DeclareMathOperator{\sgn}{sgn}
\newtheorem{theorem}{Theorem}
\newtheorem{lemma}[theorem]{Lemma}
\newtheorem{proposition}[theorem]{Proposition}
\newtheorem{corollary}[theorem]{Corollary}
\theoremstyle{remark}
\newtheorem*{remark}{Remark}
\newtheorem*{remarks}{Remarks}
\numberwithin{theorem}{section}
\numberwithin{equation}{section}

\begin{document}
\title{The Selberg trace formula as a Dirichlet series}
\author{Andrew R. Booker \and Min Lee}
\thanks{Both authors were supported by EPSRC Grants {\tt EP/H005188/1},
{\tt EP/L001454/1} and {\tt EP/K034383/1}.}
\address{
Howard House\\
University of Bristol\\
Queens Ave\\
Bristol\\
BS8 1SN\\
United Kingdom
}
\email{\tt andrew.booker@bristol.ac.uk\\min.lee@bristol.ac.uk}

\begin{abstract}
We explore an idea of Conrey and Li of expressing the Selberg trace
formula as a Dirichlet series. We describe two applications, including
an interpretation of the Selberg eigenvalue conjecture in terms of
quadratic twists of certain Dirichlet series, and a formula for an
arithmetically weighted sum of the complete symmetric square $L$-functions
associated to cuspidal Maass newforms of squarefree level $N>1$.
\end{abstract}
\maketitle

\section{Introduction}
In this paper, we explore the idea of Conrey and Li \cite{conrey-li}
(later generalized by Li in \cite{li2}) of presenting the Selberg trace
formula for Hecke operators acting on $L^2(\Gamma_0(N)\backslash\HH)$,
$N$ squarefree, as a Dirichlet series. We enhance their work in a
few ways:
\begin{itemize}
\item We prove the meromorphic continuation of the relevant Dirichlet
series to all $s\in\C$ (compared with $\Re(s)>0$ in \cite{conrey-li}).
\item We give explicit
formulas for all terms, without replacing any by estimates. Thus,
our formula entails no loss of generality, in the sense that one could
reverse the proof to derive the trace formula from it.
\item For $N>1$ we compute the trace over the \emph{newforms} of level $N$
rather than the whole spectrum. The result is a significantly
cleaner formula, though again this entails no loss of generality,
since one can recover the full formula for level $N$
by summing the formulas for newforms of levels dividing $N$.
\item We treat the Hecke operators $T_n$ for all non-zero $n$ co-prime
to the level $N$, including $n<0$. When $n<0$, there are no elliptic
terms in the trace formula, and this leads to a simpler result that is
useful for applications.
\item We base our calculations on a version of the trace formula published
by Str\"ombergsson \cite{str}, rather than working out each term from
first principles. The advantage is that Str\"ombergsson's formula has been
vetted by comparing the two sides numerically, so it is highly robust,
and this helps limit the potential for errors in the final formula. For
instance, our formula shows that the Dirichlet series we obtain can have
poles at the zeros of the scattering determinant (which are in turn
related to zeros of the Riemann zeta-function), a fact which seems to
have been overlooked in \cite{conrey-li}.
\end{itemize}

We present two applications of our formula. First, for prime $N$, we
derive a statement equivalent to Selberg's eigenvalue conjecture for
$\Gamma_0(N)$, in terms of the analytic properties of twists by the
quadratic character (mod $N$) of the family of Dirichlet series arising
from our formula for level $1$. A similar criterion was given by Li
in \cite{li3}, and in fact Li's formulation is simpler in a way since
it involves only a single Dirichlet series. However, our formulation
makes plain the fact that the passage from level $1$ to level $N$ is
essentially a quadratic twist, providing further support for the analogy
between exceptional eigenvalues and Siegel zeros.

Second, for squarefree $N>1$, we sum our formula for $T_{-n^2}$
acting on $\Gamma_0(N)\backslash\HH$ to obtain an explicit expression
for $\sum_{j=1}^\infty(-1)^{\epsilon_j}L^*(s,\Sym^2{f_j})$,
where $\{f_j\}_{j=1}^\infty$ is a complete, arithmetically normalized
sequence of Hecke--Maass newforms on $\Gamma_0(N)\backslash\HH$,
$\epsilon_j\in\{0,1\}$ is the parity of $f_j$, and
$L^*(s,\Sym^2{f_j})$ is the complete symmetric square $L$-function.
When $N=2$, the answer can be interpreted as the Rankin--Selberg
convolution of the weight $\frac12$ harmonic weak Maass form defined in
\cite{rw} with a weight 1 Eisenstein series, much like Shimura's integral
representation for the symmetric square $L$-function.  Similar formulas
have been derived for averages of $L$-functions over an $L^2$-normalized
basis (see, e.g., \cite{MR1189505}); to our knowledge, ours is the first such to
be derived from the Selberg trace formula, with arithmetic normalization.

\subsection{Notation and statement of main results}
Let $\DD$ denote the set of discriminants, that is
$$
	\DD=\{D\in\Z:D\equiv 0\mbox{ or }1\pmod*{4}\}.
$$
Any non-zero $D\in\DD$ may be expressed uniquely in the form $d\ell^2$,
where $d$ is a fundamental discriminant and $\ell>0$. We define
$\psi_D(n)=\left(\frac{d}{n/\gcd(n,\ell)}\right)$, where
$\left(\frac{\;\;}{\;\;}\right)$ denotes the Kronecker symbol.  Note that
$\psi_D$ is periodic modulo $D$, and if $D$ is fundamental then $\psi_D$
is the usual quadratic character mod $D$.  Set
$$
L(s,\psi_D)=\sum_{n=1}^{\infty}\frac{\psi_D(n)}{n^s}
\quad\mbox{for }\Re(s)>1.
$$
Then it is not hard to see that
$$
L(s,\psi_D)=L(s,\psi_d)\prod_{p\mid\ell}
\Biggl[1+\bigl(1-\psi_d(p)\bigr)\sum_{j=1}^{\ord_p(\ell)}p^{-js}\Biggr],
$$
so that $L(s,\psi_D)$ has analytic continuation to $\C$, apart from a
simple pole at $s=1$ when $D$ is a square. In particular, if $D$ is not
a square then we have
\begin{equation}\label{eqn:L1def}
L(1,\psi_D)=L(1,\psi_d)\cdot\frac1{\ell}\prod_{p\mid\ell}
\left[1+\bigl(p-\psi_d(p)\bigr)\frac{(\ell,p^\infty)-1}{p-1}\right].
\end{equation}

Our first result is the following:
\begin{theorem}\label{thm:main1}\hspace{1cm}
\begin{enumerate}
\item
For any positive integer $n$, the series
\begin{equation}\label{eqn:level1series}
\sum_{\substack{t\in\Z\\\sqrt{t^2+4n}\notin\Z}}
\frac{L(1,\psi_{t^2+4n})}{(t^2+4n)^s}
\end{equation}
has meromorphic continuation to $\C$ and is holomorphic for
$\Re(s)>0$, apart from a simple pole of residue $\sigma_{-1}(n)$ at $s=\frac12$.
\item If $n$ is a positive integer and $N$ is a prime
such that $\left(\frac{-4n}{N}\right)=-1$,
then the series
\begin{equation}\label{eqn:quadtwist}
\sum_{\substack{t\in\Z\\\sqrt{t^2+4n}\notin\Z}}
\frac{L(1,\psi_{t^2+4n})\left(\frac{t^2+4n}{N}\right)}{(t^2+4n)^s}
\end{equation}
has meromorphic continuation to $\C$ and is holomorphic for
$\Re(s)>\frac7{64}$.
\item For any prime $N$, the Selberg eigenvalue conjecture
is true for $\Gamma_0(N)$ if and only if \eqref{eqn:quadtwist} is
holomorphic on $\Re(s)>0$ for all primes $n$ satisfying
$\left(\frac{-4n}{N}\right)=-1$.
\end{enumerate}
\end{theorem}

\begin{remarks}\hspace{1pt}
\begin{enumerate}
\item
The locations and residues of the poles of \eqref{eqn:level1series}
and \eqref{eqn:quadtwist} are related to the trace of
$T_{-n}$ over the discrete spectrum of the Laplacian on
$L^2(\Gamma_0(1)\backslash\HH)$ and $L^2(\Gamma_0(N)\backslash\HH)$,
respectively. See Propositions~\ref{prop:specialization_N=1} and
\ref{prop:specialization_N>1} for full details.
\item
A simple consequence of (1) is the asymptotic
$$
\sum_{\substack{t\in\Z\cap[1,X]\\\sqrt{t^2+4n}\notin\Z}}
L(1,\psi_{t^2+4n})\sim\sigma_{-1}(n)X
\quad\mbox{as }X\to\infty.
$$
In fact, arguing as in the proof of Theorem~\ref{thm:main3} below, one
can see that the two sides are equal up to an error of
$O_{n,\varepsilon}\big(X^{\frac35+\varepsilon}\big)$.
Related averages over discriminants of the form $t^{2k}-4$ for fixed $k$
were computed by Sarnak \cite{sarnak2} and subsequently generalized by
Raulf \cite{raulf}, who obtained averages over arithmetic progressions
and also sieved to reach the fundamental discriminants. It would be
interesting to see whether our formula for the generating function could
be used in conjunction with Raulf's work to obtain sharper error terms.
(See also Hashimoto's recent improvement \cite{hashimoto} of
\cite{sarnak1} and \cite{raulf} for the closely related problem of
determining the average size of the class number over discriminants
ordered by their units.)
\end{enumerate}
\end{remarks}

\medskip
Next, we define more general versions of the coefficients $L(1,\psi_D)$
that will turn out to be related
to the newforms of a given squarefree level $N>1$. For non-zero
$D=d\ell^2\in\DD$, let $m=(N^\infty,\ell)$, and define
\begin{equation}\label{eqn:cNdef}
c_N(D)=\begin{cases}
m^{-1}\prod_{p\mid N}(\psi_{D/m^2}(p)-1)\cdot L(1,\psi_{D/m^2})
&\mbox{if }d\ne1,\\
\Lambda(N)\bigl(m^{-1}-\frac{2N}{N+1}\bigr)&\mbox{if }d=1,
\end{cases}
\end{equation}
where $\Lambda$ denotes the von Mangoldt function. For notational
convenience, we set $c_N(D)=0$ when $D\equiv 2$ or $3\pmod*{4}$.

When $N=2$, it was shown in \cite{rw} that the numbers
$$
c^+(n)=\begin{cases}
c_2(n)&\mbox{if }n\equiv0,1\pmod*{4},\\
2c_2(4n)&\mbox{if }n\equiv2,3\pmod*{4},
\end{cases}
$$
for $n>0$, are the Fourier coefficients
of a weight $\frac12$ mock modular form for $\Gamma_0(4)$
with shadow $\Theta^3$, where
$\Theta=\sum_{n\in\Z}q^{n^2}$ is the classical theta
function.\footnote{Our definition of $c^+(n)$ differs from that in
\cite{rw} in a few minor ways. First, we have scaled their definition
by the constant $\frac{\pi}6$. Second, there is a mistake in the formula
for $c^+(n)$ given in \cite{rw} for square values of $n$, to the effect
that their formula should be multiplied by $2-2^{-\ord_2(n)/2}$.  Third,
the mock modular form is only determined modulo $\C\Theta$ from its
defining properties; we add a particular multiple of $\Theta$ to make
Theorem~\ref{thm:main2} as symmetric as possible.}

Now, for a positive integer $n$, put
$$
r(n)=\tfrac12\#\bigl\{(x,y)\in\Z^2:n=x^2+4y^2\bigr\}
=
\bigl(1+\cos\tfrac{\pi{n}}2\bigr)\sum_{d\mid n}\psi_{-4}(d).
$$
The factor $\frac12$ is chosen to make
$r$ multiplicative; in fact, we have
$$
\sum_{n=1}^\infty\frac{r(n)}{n^s}=
(1-2^{-s}+2^{1-2s})\zeta(s)L(s,\psi_{-4}),
$$
so that $r(n)$ are the Fourier coefficients of a
modular form of weight $1$ and level $16$.
\begin{theorem}\label{thm:main2}
Let $N>1$ be a squarefree integer, and
let $\{f_j\}_{j=1}^\infty$ be a complete sequence of arithmetically normalized
Hecke--Maass newforms on $\Gamma_0(N)\backslash\HH$, with parities
$\epsilon_j\in\{0,1\}$, Laplace
eigenvalues $\frac14+r_j^2$ and Hecke eigenvalues $\lambda_j(n)$.
Define
$$
\Gamma_\R(s) = \pi^{-\frac{s}{2}} \Gamma\!\left(\tfrac{s}{2}\right), \quad
\zeta^*(s)=\Gamma_\R(s)\zeta(s),\quad
E_N^*(s)=N^{s/2}\prod_{p\mid N}(1-p^{-s}),\quad
\zeta_N^*(s)=E_N^*(s)\zeta^*(s),
$$
$$
L^*(s,\Sym^2{f_j})=
\frac{\Gamma_\R(s)\Gamma_\R(s-2ir_j)\Gamma_\R(s+2ir_j)}{\Gamma_\R(2s)}
\zeta_N^*(2s)\sum_{n=1}^\infty\frac{\lambda_j(n^2)}{n^s},
$$
$$
\IN(s;\sigma)=\frac1{2\pi i}
\int_{\Re(u)=-\sigma}\frac{E_N^*(s)E_N^*(1-s)}{E_N^*(u)E_N^*(1-u)}
E_N^*(2u)\zeta^*(s-u)\zeta^*(s+u)\,du,
$$
and
$$
F_N(s)=\zeta_N^*(4s)\Gamma_\R(2s)
\sum_{n=1}^\infty\frac{c_N(n)r(n)}{n^s}.
$$
Then, for any $\sigma>2$ and $s\in\C$ with
$\Re(2s)\in(2,\sigma)$, we have
$$
\begin{aligned}
F_N(s)&=\sum_{j=1}^\infty(-1)^{\epsilon_j}L^*(2s,\Sym^2{f_j})\\
&+\zeta^*(2s)\left(
\sqrt{N}\zeta_N^*(2s-1)\zeta_N^*(-2s)
-\tfrac{N\Lambda(N)}{N+1}
\bigl[\zeta_N^*(4s)+\zeta_N^*(2-4s)+\IN(2s;\sigma)\bigr]\right).
\end{aligned}
$$
In particular, $F_N(s)$ continues to an entire function,
apart from at most simple poles at $s\in\{-\frac12,0,\frac12,1\}$,
and is symmetric with respect to $s\mapsto\frac12-s$.
\end{theorem}

An analogue of Theorem~\ref{thm:main2} holds for $N=1$ as well,
but the result is more complicated to state. We content ourselves with
the following consequence.
\begin{theorem}\label{thm:main3}
As $X\to\infty$, for any $\varepsilon>0$, 
$$
	\sum_{\substack{0<D\leq X\\\sqrt{D}\notin\Z}}
	L(1, \psi_D) r(D) 
	=
	\frac{15 \zeta(3)}{4\pi} X + O\bigl(X^{\frac{8}{11}+\varepsilon}\bigr).
$$
\end{theorem}

\subsection{Outline of the paper}
In Section~\ref{sec:tf} we present the trace formula for $T_n$ acting
on $\Gamma_0(N)\backslash\HH$, $N$ squarefree, using a form of the
test function that will be convenient for later application; see
Propositions \ref{prop:tf1} and \ref{prop:tf2}.
In Section~\ref{sec:specialization}
we specialize the choice of test function as in \cite{conrey-li}, so
that the hyperbolic terms become Dirichlet series. Finally, in
Section~\ref{sec:applications} we apply the formula derived in
Section~\ref{sec:specialization} to prove
Theorems~\ref{thm:main1}--\ref{thm:main3}.

\subsection*{Acknowledgements}
We thank Dorian Goldfeld and Peter Sarnak for helpful suggestions and
corrections.

\section{The Selberg trace formula}\label{sec:tf}
Let $N$ be a squarefree positive integer, and
for any $\lambda\in\R_{\ge0}$, let $\A(N,\lambda)$ denote the space of
automorphic forms $f\in L^2(\Gamma_0(N)\backslash\HH)$ satisfying
$\left(y^2\bigl(\frac{\partial^2}{\partial x^2}
+\frac{\partial^2}{\partial y^2}\bigr)+\lambda\right)f=0$.
We begin with the trace formula for level $1$.
\begin{proposition}\label{prop:tf1}
Let $n$ be a non-zero integer and
$q:[0,\infty)\to\C$ a smooth function
satisfying $q(v)\ll(1+v)^{-\frac12-\delta}$ for some $\delta>0$.
Define
\begin{align*}
f(y)&=q\!\left(\frac{y^2+2(n-|n|)}{4|n|}\right),\\
h(r)&=\int_\R q\bigl(\sinh^2\bigl(\tfrac{u}2\bigr)\bigr)e^{iru}\,du
=2|n|^{-ir}\int_0^\infty f\!\left(v-\frac{n}{v}\right)v^{2ir-1}\,dv
\end{align*}
for $r\in\C$ with $|\Im(r)|<\frac12+\delta$,
$$
W(D)=
\begin{cases}
L(1,\psi_D)\frac{\sqrt{|D|}}{\pi}
\int_\R\frac{f(y)}{y^2+|D|}\,dy
&\mbox{if }D<0,\\
L(1,\psi_D)f\bigl(\sqrt{D}\bigr)
&\mbox{if }D>0,\;\sqrt{D}\notin\Z,\\
\sum_{m\mid\sqrt{D}}\Lambda(m)(1-m^{-1})f(\sqrt{D})
+\int_{\sqrt{D}}^\infty\frac{f(y)}{y+{\sqrt{D}}}\,dy
&\mbox{if }0\ne\sqrt{D}\in\Z,\\
(\gamma-\log2)f(0)
+\frac12\int_0^\infty\frac{f(y)+f(y^{-1})-f(0)}{y}\,dy
+\frac13\int_0^\infty\frac{f(0)-f(y)}{y^2}\,dy
&\mbox{if }D=0
\end{cases}
$$
for $D\in\DD$, and
\begin{equation}\label{eqn:Fdef}
\begin{aligned}
F(a)=2\sum_{m=1}^\infty&\frac{\Lambda(m)}{m}
f\!\left(am-\frac{n}{am}\right)
+2a\int_a^\infty\frac{f(v-\frac{n}{v})-f(a-\frac{n}{a})}{v^2-a^2}\,dv\\
&+\bigl(\gamma+\log(4\pi)\bigr)
f\!\left(a-\frac{n}{a}\right)-\frac{h(0)}{4}
\end{aligned}
\end{equation}
for $a\in\Z_{>0}$ with $a\mid n$.
Then
\begin{equation}\label{eqn:tf1}
\begin{aligned}
\sum_{\lambda\in\R_{\ge0}}\Tr T_n|_{\A(1,\lambda)}
h\Bigl(\sqrt{\lambda-\tfrac14}\Bigr)
=\sum_{\substack{a\in\Z_{>0}\\a\mid n}}F(a)+\sum_{t\in\Z}W(t^2-4n)
\end{aligned}
\end{equation}
and
\begin{equation}\label{eqn:Faint}
\sum_{\substack{a\in\Z_{>0}\\a\mid n}}F(a)
=\frac1{4\pi}\int_\R h(r)\frac{\sigma_{2ir}(|n|)}{|n|^{ir}}
\frac{\phi'}{\phi}\!\left(\frac12+ir\right)dr
+\frac{\sigma_0(|n|)}4h(0),
\end{equation}
where $\phi(s)=\frac{\zeta^*(2(1-s))}{\zeta^*(2s)}$.
\end{proposition}
\begin{proof}
We first derive \eqref{eqn:Faint}.
For a sufficiently nice, even Fourier transform pair $g,h$,
it was shown in \cite[p.~509]{hejhal2} that
$$
\frac1{4\pi}\int_\R h(r)
\frac{\phi'}{\phi}\!\left(\frac12+ir\right)dr
=g(0)\log\pi-\frac1{2\pi}\int_\R h(r)\psi\bigl(\tfrac12+ir\bigr)\,dr
-\frac{h(0)}2+2\sum_{m=1}^\infty\frac{\Lambda(m)}{m}g(2\log{m}).
$$
Replacing $g(u)$ by
$\sum_{\substack{ad=n\\a>0}}g\bigl(u-\log\bigl|\tfrac{a}{d}\bigr|\bigr)$
and $h(r)$ by $h(r)\sum_{\substack{ad=n\\a>0}}\bigl|\tfrac{a}{d}\bigr|^{ir}$,
we get
\begin{equation}\label{eqn:withpsiint}
\begin{aligned}
\frac1{4\pi}\int_\R h(r)&\frac{\sigma_{2ir}(|n|)}{|n|^{ir}}
\frac{\phi'}{\phi}\!\left(\frac12+ir\right)dr
+\frac{\sigma_0(|n|)}4h(0)\\
&=\sum_{\substack{ad=n\\a>0}}\Biggl[
g\!\left(\log\left|\frac{a}{d}\right|\right)\log\pi
-\frac1{2\pi}\int_\R h(r)\left|\frac{a}{d}\right|^{ir}
\psi\bigl(\tfrac12+ir\bigr)\,dr\\
&\qquad\qquad-\frac{h(0)}4
+2\sum_{m=1}^\infty\frac{\Lambda(m)}{m}
g\!\left(\log\left|\frac{a}{d}\right|-2\log{m}\right)\Biggr].
\end{aligned}
\end{equation}
Similarly, from the identity
$$
-\frac1{2\pi}\int_{-\infty}^{\infty}
h(r)\psi\bigl(\tfrac12+ir\bigr)\,dr
=g(0)\log\bigl(4e^{\gamma}\bigr)
+\int_0^{\infty}\frac{g(u)-g(0)}{2\sinh(u/2)}\,du
$$
we derive
$$
\begin{aligned}
-\sum_{\substack{ad=n\\a>0}}&\frac1{2\pi}\int_{-\infty}^{\infty}
h(r)\left|\frac{a}{d}\right|^{ir}
\psi\big(\tfrac12+ir\big) dr\\
&=\sum_{\substack{ad=n\\a>0}}
\left[g\!\left(\log\left|\frac{a}{d}\right|\right)\log\bigl(4e^{\gamma}\bigr)
+\int_0^{\infty}
\frac{g\!\left(u+\log\left|\frac{a}{d}\right|\right)
-g\!\left(\log\left|\frac{a}{d}\right|\right)}
{2\sinh(u/2)}\,du\right],
\end{aligned}
$$
so that
\begin{align*}
\frac1{4\pi}\int_\R h(r)&\frac{\sigma_{2ir}(|n|)}{|n|^{ir}}
\frac{\phi'}{\phi}\!\left(\frac12+ir\right)dr
+\frac{\sigma_0(|n|)}4h(0)
\\
&=\sum_{\substack{ad=n\\a>0}}\Biggl[
g\!\left(\log\left|\frac{a}{d}\right|\right)\log\bigl(4\pi e^\gamma\bigr)
+\int_0^{\infty}
\frac{g\!\left(u+\log\left|\frac{a}{d}\right|\right)
-g\!\left(\log\left|\frac{a}{d}\right|\right)}
{2\sinh(u/2)}\,du\\
&\qquad\qquad-\frac{h(0)}4
+2\sum_{m=1}^\infty\frac{\Lambda(m)}{m}
g\!\left(\log\left|\frac{a}{d}\right|-2\log{m}\right)\Biggr].
\end{align*}
Now let $g(u)=q(\sinh^2(\tfrac{u}2))$.
Then $g\!\left(u+\log\left|\frac{a}{d}\right|\right)
=f\big(ae^{\frac{u}2}-de^{-\frac{u}2}\big)$, so on making the
substitution $v=ae^{u/2}$, we get $\sum_{\substack{ad=n\\a>0}}F(a)$,
as required.

Turning to \eqref{eqn:tf1},
in \cite[\S2.1]{str} we find the following
trace formula for level $1$:
\begin{equation}\label{eqn:trace1}
\begin{aligned}
&\sum_{\lambda\in\R_{\ge0}}\Tr T_n|_{\A(1,\lambda)}
h\Bigl(\sqrt{\lambda-\tfrac14}\Bigr)\\
&=\sum_{\substack{t\in\Z\\\sqrt{t^2-4n}\notin\Z}} 
\biggl(\sum_{c\mid\ell}\hh(\rr[c])
\left[\rr[1]^\times:\rr[c]^\times\right] 
\biggr)A(t,n)\\
&+\begin{cases}
\frac1{12\sqrt{n}}\int_\R r\tanh(\pi r)h(r)\,dr
+g(0)\log\frac{\pi\sqrt{n}}2+\frac{h(0)}4
-\frac1{2\pi}\int_\R h(r)\psi(1+ir)\,dr
&\mbox{if }\sqrt{n}\in\Z,\\
0&\mbox{otherwise}
\end{cases}\\
&+\sum_{\substack{a\in\Z_{>0}\\a\mid n,\,a^2\ne n}}\Biggl\{
\left[\log\pi+\log\left|a-\frac{n}{a}\right|
-X\!\left(\left|a-\frac{n}{a}\right|\right)\right]
g\!\left(\log\left|\frac{a^2}{n}\right|\right)\\
&\qquad\qquad+\frac12\int_{\bigl|\log\frac{a}{\sqrt{|n|}}\bigr|}^\infty 
g(u)\frac{e^{\frac{u}2}+\sgn(n)e^{-\frac{u}2}}
{e^{\frac{u}2}-\sgn(n)e^{-\frac{u}2}
+\Bigl|\frac{a}{\sqrt{|n|}}-\sgn(n)\frac{\sqrt{|n|}}{a}\Bigr|}\,du\Biggr\}\\
&+\sum_{\substack{a\in\Z_{>0}\\a\mid n}}\biggl[
2\sum_{m=1}^\infty\frac{\Lambda(m)}{m} 
g\!\left(\log\left|\frac{a^2}{n}\right|-2\log m\right)-\frac{h(0)}4
-\frac1{2\pi}\int_{-\infty}^\infty h(r)\left|\frac{a^2}{n}\right|^{ir} 
\psi\bigl(\tfrac12+ir\bigr)\,dr\biggr],
\end{aligned}
\end{equation}
where the notation is as follows:
\begin{itemize}
\item $t^2-4n=d\ell^2$, where $d$ is a fundamental discriminant
and $\ell>0$;
\item $\rr[c]=\Z+\Z c\frac{d+\sqrt{d}}2$ is the quadratic
order of conductor $c$ in $\Q(\sqrt{d})$ and
$\rr[c]^\times$ is its unit group;
\item $\hh(\rr[c])=\frac{\hh(\rr[1])c\prod_{p\mid c}(1-\psi_d(p)p^{-1})}
{[\rr[1]^\times:\rr[c]^\times]}$
is the narrow class number of $\rr[c]$;
\item $\epsilon_d^+>1$ is the smallest unit in $\rr[1]^\times$
with norm $1$ (i.e.\
the fundamental unit when it has norm $1$ and its square otherwise);
\item
$A(t,n)=\begin{cases}
\frac{\log\epsilon_d^+}{\sqrt{t^2-4n}}
g\!\left(2\log\frac{|t|+\sqrt{t^2-4n}}{2\sqrt{|n|}}\right)
&\mbox{if }t^2-4n>0,\\
\frac2{|\rr[1]^\times|\sqrt{4n-t^2}}
\int_\R\frac{e^{-2r\arccos(|t|/2\sqrt{|n|})}}{1+e^{-2\pi r}}h(r)\,dr
&\mbox{if }t^2-4n<0;
\end{cases}$
\item $X(u)=\frac1u\sum_{m\pmod*{u}}\log\gcd(m,u)
=\sum_{m\mid u}\frac{\Lambda(m)}{m}$.
\end{itemize}

Writing $D=t^2-4n=d\ell^2$, we have
\begin{align*}
&\hh(\rr[c])\left[\rr[1]^\times:\rr[c]^\times\right]A(t,n)\\
&=\frac{c}{\ell}
\prod_{p\mid c}\left(1-\frac{\psi_d(p)}p\right)
\begin{cases}
\frac{\hh(\rr[1])\log\epsilon_d^+}{\sqrt{d}}
g\!\left(2\log\frac{|t|+\sqrt{D}}{2\sqrt{|n|}}\right)
&\mbox{if }D>0,\\
\frac{2\hh(\rr[1])}{|\rr[1]^\times|\sqrt{|d|}}
\int_\R\frac{e^{-2r\arccos(|t|/2\sqrt{|n|})}}{1+e^{-2\pi r}}h(r)\,dr
&\mbox{if }D<0.
\end{cases}
\end{align*}
By Dirichlet's class number formula, we have
$$
L(1,\psi_d)=
\begin{cases}
\frac{\hh(\rr[1])\log\epsilon_d^+}{\sqrt{d}}
&\mbox{if }d>0,\\
\frac{2\pi\hh(\rr[1])}{|\rr[1]^\times|\sqrt{|d|}}
&\mbox{if }d<0,
\end{cases}
$$
so this becomes
$$
L(1,\psi_d)
\frac{c}{\ell}\prod_{p\mid c}\left(1-\frac{\psi_d(p)}p\right)
\begin{cases}
g\!\left(2\log\frac{|t|+\sqrt{D}}{2\sqrt{|n|}}\right)
&\mbox{if }t^2-4n>0,\\
\frac1\pi
\int_\R\frac{e^{-2r\arccos(|t|/2\sqrt{|n|})}}{1+e^{-2\pi r}}h(r)\,dr
&\mbox{if }t^2-4n<0.
\end{cases}
$$
Summing over $c$ and using \eqref{eqn:L1def}, we find by a short
computation that
\begin{equation}\label{eqn:L1psiD}
\sum_{c\mid\ell}L(1,\psi_d)
\frac{c}{\ell}\prod_{p\mid c}\left(1-\frac{\psi_d(p)}p\right)
=L(1,\psi_D).
\end{equation}
Further, we have
$g(u)=f\big(v-\frac{n}v\bigr)$, where $v=\sqrt{|n|}e^{\frac{u}2}$.
Hence
$$
g\!\left(2\log\frac{|t|+\sqrt{D}}{2\sqrt{|n|}}\right)
=f\big(\sqrt{D}\big).
$$

Next we evaluate the integral
$\frac1\pi\int_\R
\frac{e^{-2\arccos(|t|/2\sqrt{|n|})r}}{1+e^{-2\pi r}}h(r)\,dr$. It
occurs only when $D=t^2-4n<0$, so we may assume that $n$ is positive.
Writing $\alpha=\arccos(|t|/2\sqrt{n})$, we have
\begin{align*}
\frac1\pi\int_\R&\frac{e^{-2\alpha r}}{1+e^{-2\pi r}}h(r)\,dr
=\frac1{2\pi}\int_\R\frac{e^{(\pi-2\alpha)r}}{\cosh(\pi r)}h(r)\,dr
=\frac1{2\pi}\int_\R\frac{e^{(\pi-2\alpha)r}}{\cosh(\pi r)}
\int_\R g(u)e^{iru}\,du\,dr\\
&=\frac1{2\pi}\int_\R g(u)
\int_\R\frac{e^{ir(u+i[2\alpha-\pi])r}}{\cosh(\pi r)}\,dr\,du
=\frac1{2\pi}\int_\R\frac{g(u)}{\cosh(\frac{u}2+i[\alpha-\frac{\pi}2])}
\,du\\
&=\frac1{2\pi}\int_\R\frac{g(u)}
{\cosh(\frac{u}2)\sin\alpha-i\sinh(\frac{u}2)\cos\alpha}
\,du\\
&=\frac1{2\pi}\int_\R g(u)
\frac{\cosh(\frac{u}2)\sin\alpha+i\sinh(\frac{u}2)\cos\alpha}
{\sinh^2(\frac{u}2)+\sin^2\alpha}\,du
=\frac1{2\pi}\int_\R\frac{g(u)\cosh(\frac{u}2)\sin\alpha}
{\sinh^2(\frac{u}2)+\sin^2\alpha}\,du,
\end{align*}
where in the last line we make use of the fact that $g$ is even.
Writing $g(u)=q\big(\sinh^2(\frac{u}2)\big)$ and making the
substitution $y=2\sqrt{n}\sinh\frac{u}2$, this becomes simply
$$
\frac{\sqrt{|D|}}\pi\int_\R q\!\left(\frac{y^2}{4n}\right)\frac{dy}{y^2+|D|}
=\frac{\sqrt{|D|}}\pi\int_\R\frac{f(y)}{y^2+|D|}\,dy.
$$
Hence, altogether we have
\begin{equation}\label{eqn:hypsum}
\sum_{c\mid\ell}\hh(\rr[c])
\left[\rr[1]^\times:\rr[c]^\times\right]A(t,n)
=L(1,\psi_D)
\begin{cases}
f\bigl(\sqrt{D}\bigr)&\mbox{if }D>0,\\
\frac{\sqrt{|D|}}\pi\int_\R f(y)\frac{dy}{y^2+|D|}
&\mbox{if }D<0.
\end{cases}
\end{equation}

Next, in the penultimate line of \eqref{eqn:trace1}, we write
$y=\sqrt{|n|}\big(e^{\frac{u}2}-(\sgn{n})e^{-\frac{u}2}\big)$,
so that $g(u)=f(y)$ and
\begin{equation}\label{eqn:ell1}
\frac12\int_{\bigl|\log\frac{a}{\sqrt{|n|}}\bigr|}^\infty 
g(u)\frac{e^{\frac{u}2}+\sgn(n)e^{-\frac{u}2}}
{e^{\frac{u}2}-\sgn(n)e^{-\frac{u}2}
+\Bigl|\frac{a}{\sqrt{|n|}}-\sgn(n)\frac{\sqrt{|n|}}{a}\Bigr|}\,du
=\int_\ell^\infty\frac{f(y)}{\ell+y}\,dy,
\end{equation}
where $\ell=|a-n/a|$,
This term contributes whenever $a^2\ne n$, and those $a$ are in
one-to-one correspondence with the non-zero square values
$D=\ell^2$ in \eqref{eqn:tf1}.
Similarly, we get a contribution of
\begin{equation}\label{eqn:ell2}
\bigl[\log\pi+\log\ell-X(\ell)\bigr]
g\!\left(\log\left|\frac{a^2}{n}\right|\right)
=\Biggl(\log\pi+\sum_{m\mid\ell}\Lambda(m)\bigl(1-m^{-1}\bigr)\Biggr)f(\ell)
\end{equation}
when $D=\ell^2\ne0$.
As for the final line of \eqref{eqn:trace1}, by \eqref{eqn:withpsiint}
and \eqref{eqn:Faint}, it is
$$
\sum_{0<a\mid n}\left[F(a)-
g\!\left(\log\left|\frac{a^2}{n}\right|\right)\log\pi\right]
=\sum_{0<a\mid n}\left[F(a)-f\!\left(a-\frac{n}a\right)\log\pi\right],
$$
and together with \eqref{eqn:ell1} and \eqref{eqn:ell2} we get the
contributions from the sum over $a$ and the non-zero square values of
$D$ in \eqref{eqn:tf1}.

Finally, the terms of \eqref{eqn:trace1} with $\sqrt{n}\in\Z$ correspond to
$D=0$, and they clearly occur only when $n$ is positive.
For any $c>0$, we have
\begin{align*}
g(0)&\log\frac{\pi\sqrt{n}}2+\frac{h(0)}4
-\frac1{2\pi}\int_\R h(r)\psi(1+ir)\,dr\\
&=g(0)\log\frac{\pi e^\gamma\sqrt{n}}2
+\int_0^\infty\log(2\sinh(u/2))g'(u)\,du\\
&=g(0)\log\frac{\pi e^\gamma\sqrt{n}}2
+g(0)\log\!\left(2\sinh{\frac{c}2}\right)
+\int_0^c\frac{g(u)-g(0)}{2\tanh\frac{u}2}\,du
+\int_c^\infty\frac{g(u)}{2\tanh\frac{u}2}\,du.
\end{align*}
Choosing $c$ such that $2\sqrt{n}\sinh\frac{c}2=1$ and making the
substitution $y=2\sqrt{n}\sinh\frac{u}2$, this becomes
\begin{align*}
f(0)\log\frac{\pi e^\gamma}2
&+\int_0^1\frac{f(y)-f(0)}y\,dy
+\int_1^\infty\frac{f(y)}y\,dy\\
&=f(0)\log\frac{\pi e^\gamma}2
+\frac12\int_0^\infty\frac{f(y)+f(y^{-1})-f(0)}y\,dy.
\end{align*}
Similarly, we have
\begin{equation}\label{eqn:idterm}
\begin{aligned}
\frac1{12\sqrt{n}}\int_\R r\tanh(\pi r)h(r)\,dr
&=-\frac1{12\sqrt{n}}\int_\R\frac{g'(u)}{\sinh(u/2)}\,du
=\frac1{12\sqrt{n}}\int_0^\infty\frac{g(0)-g(u)}{\sinh(u/2)\tanh(u/2)}\,du\\
&=\frac16\int_\R\frac{f(0)-f(y)}{y^2}\,dy.
\end{aligned}
\end{equation}
\end{proof}

Next, suppose that $N>1$. In this case it is helpful to restrict
the trace formula to the newforms of level $N$. To be precise,
if $M_1,M_2$ are positive integers such that $M_1M_2\mid N$ and
$M_1\ne N$, then there is a linear map
$L_{M_1,M_2}:\A(M_1,\lambda)\to \A(N,\lambda)$
which sends $f\in\A(M_1,\lambda)$
to the function $z\mapsto f(M_2z)$.
Let $\Anew(N,\lambda)\subseteq\A(N,\lambda)$
denote the ``new'' subspace of forms that are orthogonal
(with respect to the Petersson inner product) to the images of
$L_{M_1,M_2}$ for all $M_1,M_2$.

For $D\in\Z$, let
\begin{equation}\label{eqn:cN0def}
	c_N^\circ(D)
	=
	\begin{cases}
	\frac{\varphi(N)}6&\mbox{if }D=0,\\
	\frac{\Lambda(N)}{(\ell,N^\infty)}&\mbox{if }D=\ell^2\ne0,\\
	c_N(D)&\mbox{otherwise},
	\end{cases}
\end{equation}
where $c_N$ is as defined in \eqref{eqn:cNdef}.
Then the trace formula for level $N$ is as follows.
\begin{proposition}\label{prop:tf2}
Let $n$, $f$ and $h$ be as in Proposition~\ref{prop:tf1},
and let $N>1$ be a squarefree integer with $(n,N)=1$.  Then
\begin{equation}\label{eqn:tf2}
\begin{aligned}
\sum_{\lambda\in\R_{>0}}
&\Tr T_n|_{\Anew(N,\lambda)}
h\Bigl(\sqrt{\lambda-\tfrac14}\Bigr)\\
&=\sum_{\substack{t\in\Z\\D=t^2-4n}}
c_N^\circ(D)
\begin{cases}
f\bigl(\sqrt{D}\bigr)&\mbox{if }D>0,\\
	\tfrac{\sqrt{|D|}}{\pi} \int_{\R} \frac{f(y)}{y^2+|D|} \, dy
&\mbox{if }D<0,\\
\int_\R\frac{f(0)-f(y)}{y^2}\,dy
&\mbox{if }D=0
\end{cases}\\
&\qquad-\mu(N)\frac{\sigma_1(|n|)}{\sqrt{|n|}}
h\!\left(\frac{i}2\right)
-2\Lambda(N)\sum_{\substack{a\in\Z_{>0}\\a\mid n}}
\sum_{r=0}^\infty N^{-r}f\!\left(aN^r-\frac{n}{aN^r}\right).
\end{aligned}
\end{equation}
\end{proposition}
\begin{proof}
Specializing the formula in \cite[\S2.2]{str} to trivial nebentypus character,
we have the following trace formula for newforms on $\Gamma_0(N)$, with
notation as in \eqref{eqn:trace1}:
\begin{equation}\label{eqn:traceN}
\begin{aligned}
&\frac{\mu(N)\sigma_1(|n|)}{\sqrt{|n|}}h\!\left(\frac{i}2\right)
+\sum_{\lambda\in\R_{>0}}\Tr T_n|_{\Anew(N,\lambda)}
h\Bigl(\sqrt{\lambda-\tfrac14}\Bigr)\\
&=\sum_{\substack{t\in\Z\\\sqrt{t^2-4n}\notin\Z}} 
\biggl(\sum_{\substack{c\mid\ell\\(c,N)=1}}\hh(\rr[c])
\left[\rr[1]^\times:\rr[c]^\times\right]
\prod_{p\mid N}\left[\left(\frac{d}{p}\right)-1\right]
\biggr)A(t,n)\\
&+\begin{cases}
\frac{\varphi(N)}{12\sqrt{n}}\int_\R r\tanh(\pi r)h(r)\,dr
&\mbox{if }\sqrt{n}\in\Z,\\
0&\mbox{otherwise}
\end{cases}\\
&+\Lambda(N)\sum_{\substack{a\in\Z_{>0}\\a\mid n,\,a^2\ne n}}
\frac{g\!\left(\log\left|\frac{a^2}{n}\right|\right)}{(N^\infty,|a-n/a|)}
-2\Lambda(N)\sum_{\substack{a\in\Z_{>0}\\a\mid n}}\sum_{r=0}^\infty
N^{-r}g\!\left(\log\left|\frac{a^2}{n}\right|-2r\log N\right).
\end{aligned}
\end{equation}
Applying \eqref{eqn:L1psiD} with $D$ replaced by $D/(N^\infty,\ell)$
and comparing to the definition \eqref{eqn:cNdef},
we find that
$$
\sum_{\substack{c\mid\ell\\(c,N)=1}}L(1,\psi_d)
\frac{c}{\ell}\prod_{p\mid c}\left(1-\frac{\psi_d(p)}p\right)
\cdot\prod_{p\mid N}\left[\left(\frac{d}{p}\right)-1\right]
=c_N^\circ(D).
$$
Hence, following the derivation of \eqref{eqn:hypsum}, we get
\begin{align*}
\sum_{\substack{c\mid\ell\\(c,N)=1}}\hh(\rr[c])
&\left[\rr[1]^\times:\rr[c]^\times\right]
\prod_{p\mid N}\left[\left(\frac{d}{p}\right)-1\right]
A(t,n)\\
&=c_N^\circ(D)\begin{cases}
f\bigl(\sqrt{D}\bigr)&\mbox{if }D>0,\\
	\tfrac{\sqrt{|D|}}{\pi} \int_{\R} \frac{f(y)}{|D|+y^2} \, dy
&\mbox{if }D<0.
\end{cases}
\end{align*}
When $n$ is a square,
the corresponding term of \eqref{eqn:traceN} is,
by \eqref{eqn:idterm},
$$
\frac{\varphi(N)}6\int_\R\frac{f(0)-f(y)}{y^2}\,dy,
$$
and this matches the contribution to \eqref{eqn:tf2} from $D=0$.
Similarly, the terms of \eqref{eqn:tf2} corresponding to $D=\ell^2\ne0$
match the first sum on the last line of \eqref{eqn:traceN}. 
\end{proof}

\section{Specialization of the test function}\label{sec:specialization}
In this section, we compute the terms of Propositions~\ref{prop:tf1}
and \ref{prop:tf2} explicitly for $q(v)=[4(v+1)]^{-s}$.
We change notation slightly, replacing $n$ by $\pm n$,
where $n\in\Z_{>0}$.
\begin{proposition}\label{prop:specialization_N=1}
Let $s\in\C$ with $\Re(s)>\frac12$ and $n\in\Z_{>0}$.
Define
$$
\Phi(x,s)=\begin{cases}
\frac1{s\Beta(s,\frac12)}&\mbox{if }x=0,\\
x^{-s}I_x\bigl(s,\tfrac12\bigr)&\mbox{if }0<x<1,\\
x^{-s}&\mbox{if }x\ge 1
\end{cases}
$$
and
$$
	\Psi(x, s) 
	=
	\int_{\sqrt{x-1}}^\infty \frac{(y^2+1)^{-s}}{y+\sqrt{x-1}} \, dy, 
$$
where $\Beta(a,b)=\frac{\Gamma(a)\Gamma(b)}{\Gamma(a+b)}$ denotes
the Euler Beta-function and
$$
I_x(a,b)=\frac1{\Beta(a,b)}\int_0^x t^{a-1}(1-t)^{b-1}\,dt
$$
is the normalized incomplete Beta-function.
Then
\begin{equation}\label{eqn:level1plus}
\begin{aligned}
	&\sum_{\substack{\lambda\in\R_{>0}\\\lambda=\frac14+r^2}}	
	\Tr T_n|_{\A(1,\lambda)}\Beta(s+ir,s-ir)
	+
	\frac{\sigma_{1}(n)}{\sqrt{n}} 
	\Beta\!\left(s-\frac{1}{2}, s+\frac{1}{2}\right)
	\\
	&-\frac1{4\pi}\int_\R\Beta(s-ir,s+ir)
\frac{\sigma_{2ir}(n)}{n^{ir}}
\frac{\phi'}{\phi}\!\left(\frac12+ir\right)dr
-\frac{\sigma_0(n)}4\Beta(s,s)\\
&=4^{-s}\sum_{\substack{t\in\Z\\D=t^2-4n}}
\begin{cases}
	L(1,\psi_D)\Phi\bigl(\frac{t^2}{4n},s\bigr)
		&\mbox{if }\sqrt{D}\notin\Z,\\
	\sum_{m\mid \sqrt{D}} \Lambda(m)(1-m^{-1}) \Phi\bigl(\frac{t^2}{4n}, s\bigr) + \Psi\bigl(\frac{t^2}{4n}, s\bigr)
		&\mbox{if }0\ne\sqrt{D}\in\Z,\\
	\tfrac12\bigl(\psi(s)+\gamma+\log{n}\bigr) +\tfrac16\sqrt{\frac{\pi}n}\frac{\Gamma(s+\frac12)}{\Gamma(s)}
		&\mbox{if }D=0,
\end{cases}
\end{aligned}
\end{equation}
and
\begin{equation}\label{eqn:level1minus}
\begin{aligned}
	&\sum_{\substack{\lambda\in\R_{>0}\\\lambda=\frac14+r^2}}
	\Tr T_{-n}|_{\A(1,\lambda)}\Beta(s+ir,s-ir)
	+
	\frac{\sigma_{1}(n)}{\sqrt{n}} 
	\Beta\!\left(s-\frac{1}{2}, s+\frac{1}{2}\right)
	\\
	&-\frac1{4\pi}\int_\R\Beta(s-ir,s+ir)
	\frac{\sigma_{2ir}(n)}{n^{ir}}
	\frac{\phi'}{\phi}\!\left(\frac12+ir\right)dr
	-\frac{\sigma_0(n)}4\Beta(s,s)\\
	&=\sum_{\substack{t\in\Z\\D=t^2+4n}}\left(\frac{n}{D}\right)^s
\begin{cases}
L(1,\psi_D)&\mbox{if }\sqrt{D}\notin\Z,\\
\sum_{m\mid\sqrt{D}}\Lambda(m)\bigl(1-m^{-1}\bigr)
+\tfrac12\bigl(\psi(s+\tfrac12)-\psi(s)\bigr)
&\mbox{if }\sqrt{D}\in\Z.
\end{cases}
\end{aligned}
\end{equation}
Both \eqref{eqn:level1plus} and \eqref{eqn:level1minus} continue
to meromorphic functions on $\C$ and are holomorphic for
$\Re(s)>0$, apart from simple poles of residue
$\frac{\sigma_1(n)}{\sqrt{n}}$ at $s=\tfrac12$.
\end{proposition}
\begin{proof}
With $q(v)=\bigl[4(1+v)\bigr]^{-s}$ we have
\begin{equation}\label{eqn:hs}
\begin{aligned}
h(r)&=\int_\R q\bigl(\sinh^2\bigl(\tfrac{u}2\bigr)\bigr)e^{iru}\,du
=\int_\R\bigl(2\cosh\tfrac{u}2\bigr)^{-2s}e^{iru}\,du
=\int_\R e^{u(s+ir)}\bigl(e^u+1\bigr)^{-2s}\,du\\
&=\int_0^\infty x^{s+ir}(x+1)^{-2s}\,\frac{dx}x
=\Beta(s+ir,s-ir),
\end{aligned}
\end{equation}
by \cite[3.194(3)]{GR}. 
By Stirling's formula, for any compact set $K\subset\C$ that omits all poles of
$\Beta(s+ir,s-ir)$, the estimate
$|h(r)| = \left|\Beta(s+ir, s-ir)\right| \ll e^{-\pi|r|}$
holds uniformly for $s\in K$.
Hence this is a suitable choice of test function for any fixed $s$ with
$\Re(s)>\frac12$.
Further, when combined with the Weyl-type estimate
$$
\sum_{\substack{\lambda=\frac14+r^2\in\R_{>0}\\|r|\le T}}
	\big|\Tr T_{\pm n}|_{\A(1,\lambda)}\big|\ll_n T^2,
$$
we see that the sums on the left-hand sides of \eqref{eqn:level1plus} and
\eqref{eqn:level1minus} continue to meromorphic functions of $s\in\C$.
By \eqref{eqn:Faint},
\eqref{eqn:level1plus} and \eqref{eqn:level1minus} are
$$
\sum_{\lambda\in\R_{\ge0}}\Tr T_{\pm n}|_{\A(1,\lambda)}
h\Bigl(\sqrt{\lambda-\tfrac14}\Bigr)
-\sum_{a\mid n}F(a),
$$
and it remains to evaluate $\sum_{t\in\Z}W(t^2\mp4n)$.

Let us first consider \eqref{eqn:level1minus}. Then
$f(y)=n^s|y|^{-2s}$, and we have $D=t^2+4n>0$.
Making the substitution $y=\sqrt{D/x}$, we get
\begin{align*}
\int_{\sqrt{D}}^\infty\frac{y^{-2s}}{y+\sqrt{D}}\,dy
=\frac{D^{-s}}2\int_0^1\frac{x^{s-1}}{1+\sqrt{x}}\,dx
=\frac{D^{-s}}2\int_0^1\frac{x^{s-1}-x^{s-\frac12}}{1-x}\,dx
=\frac{D^{-s}}2\bigl(\psi(s+\tfrac12)-\psi(s)\bigr),
\end{align*}
by \cite[8.361(4)]{GR}.
This yields the right-hand side of \eqref{eqn:level1minus}.

Next we consider \eqref{eqn:level1plus}, in which case
$f(y)=n^s(y^2+4n)^{-s}$ and we have $D=t^2-4n$. 
For $D<0$, 
\begin{multline*} 
	\int_\R \frac{f(y)}{y^2+|D|} \, dy
	=
	n^s 2 \int_0^\infty \frac{(y^2+4n)^{-s}}{y^2+|D|} \, dy
	=
	n^s \int_0^\infty (y+4n)^{-s} (y+|D|)^{-1} y^{-\frac{1}{2}} \, dy
	\\
	=
	n^s (4n)^{-s} |D|^{-\frac{1}{2}} 
	\Beta\!\left(\frac{1}{2}, s+\frac{1}{2}\right)
	{}_2F_1\!\left(s, \frac{1}{2}; s+1; 1-\frac{|D|}{4n}\right)
	\\
	=
	n^s (4n)^{-s} |D|^{-\frac{1}{2}} 
	\Beta\!\left(\frac{1}{2}, s+\frac{1}{2}\right)
	s
	\Beta_{\frac{t^2}{4n}}\!\left(s, \frac{1}{2}\right)\left(\frac{t^2}{4n}\right)^{-s},  
\end{multline*}
by \cite[3.197(1)]{GR}. 
Since
$$
	s\Beta\!\left(\frac{1}{2}, s+\frac{1}{2}\right)
	=
	\frac{\sqrt{\pi} \Gamma(s+\frac{1}{2})}{\Gamma(s+1)}
	=
	\frac{\pi\Gamma(s+\frac{1}{2})}{\Gamma(s)\Gamma(\frac{1}{2})}
	=
	\frac{\pi}{\Beta(s, \frac{1}{2})}, 
$$
we obtain
\begin{align}\label{e:fy_|D|}
	\frac{\sqrt{|D}}{\pi}
	\int_\R \frac{f(y)}{y^2+|D|} \, dy
	=
	4^{-s}
	I_{\frac{t^2}{4n}}\!\left(s, \frac{1}{2}\right)\left(\frac{t^2}{4n}\right)^{-s}
	.
\end{align}

For a sufficiently small $\varepsilon>0$, we have
\begin{equation}\label{e:fy}
	f(y) = 
	n^s(y^2+4n)^{-s}
	=
	y^{-2s} n^s 
	\frac{1}{2\pi i} \int_{\Re(u) =-\varepsilon} 
	\Beta(-u, s+u)
	y^{-2u} (4n)^{u}
	\, du
	,
\end{equation}
by \cite[6.422(3)]{GR}. 
Hence, for $0\neq \sqrt{D}\in \Z$, 
\begin{multline*}
	\int_{\sqrt{D}}^\infty \frac{f(y)}{y+\sqrt{D}} \, dy
	=
	n^s
	\frac{1}{2\pi i} \int_{\Re(u) =-\varepsilon} 
	\Beta(-u,s+u)
	\int_{\sqrt{D}}^\infty 
	\frac{y^{-2u-2s} }{y+\sqrt{D}} \, dy
	\, 
	(4n)^{u}
	\, du
	\\
	=
	2^{-2s-1}
	\frac{1}{2\pi i} \int_{\Re(u) = -\varepsilon} \Beta(-u, s+u)
	\left(\frac{4n}{D}\right)^{s+u}
	\left(\psi\!\left(u+s+\frac{1}{2}\right)-\psi(u+s)\right)
	du
	.
\end{multline*}
For $D=0$, 
\begin{multline*}
	\int_0^\infty \frac{f(0)-f(y)}{y^2} \, dy
	=
	n^s 
	\int_0^\infty \frac{(4n)^{-s} - (y^2+4n)^{-s}}{y^2} \, dy
	=
	\frac{4^{-s}}{\sqrt{4n}}
	\int_0^\infty \frac{1- (y^2+1)^{-s}}{y^2} \, dy
	\\
	=
	\frac{4^{-s}}{\sqrt{4n}}
	\left\{
	\left[-y^{-1}\left(1-(y^2+1)^{-s}\right)\right]_0^\infty 
	+ 2s\int_0^\infty (y^2+1)^{-s-1} \, dy
	\right\}
	.
\end{multline*}
By \cite[3.251(2)]{GR}, 
$$
	\int_0^\infty (y^2+1)^{-s-1} \, dy
	=
	\frac{1}{2} \Beta\!\left(\frac{1}{2}, s+\frac{1}{2}\right),
$$
so that
\begin{equation}\label{e:f0fy}
	\int_0^\infty \frac{f(0)-f(y)}{y^2} \, dy
	=
	\frac{4^{-s}}{\sqrt{4n}}
	s \Beta\!\left(\frac{1}{2}, s+\frac{1}{2}\right)
	=
	\frac{4^{-s}\sqrt{\pi}}{\sqrt{4n}}
	\frac{\Gamma(s+\frac{1}{2})}{\Gamma(s)}
	.
\end{equation}

Next, we have
\begin{multline*}
	\int_0^\infty \frac{f(y)+f(y^{-1})-f(0)}{y} \, dy
	=
	2 \int_0^1 \frac{f(y)+f(y^{-1})-f(0)}{y} \, dy
	\\
	=
	2 \int_0^1 \int_0^1 f'(ty) \, dt \, dy 
	+ 2\int_0^1 \frac{f(y^{-1})}{y} \, dy
	.
\end{multline*}
By \eqref{e:fy}, we have
\begin{multline*}
	f'(ty)
	=
	n^s 
	\frac{1}{2\pi i} \int_{\Re(u) =-\varepsilon} 
	\Beta(-u, s+u)
	(-2s-2u)
	(ty)^{-2s-2u-1} (4n)^{u}
	\, du
	\\
	=
	n^s 
	\frac{1}{2\pi i} \int_{\Re(u) =-\varepsilon-\Re(s)} 
	\Beta(-u, s+u)
	(-2s-2u)
	(ty)^{-2s-2u-1} (4n)^{u}
	\, du
	, 
\end{multline*}
upon moving the line of integration to $\Re(u) = -\varepsilon-\Re(s)$, so that
$-\Re(u)-\Re(s)>0$.
Therefore,
\begin{multline*}
	2\int_0^1 \int_0^1 f'(ty) \, dt \, dy 
	=
	n^s(4n)^{-s} 
	\frac{1}{2\pi i} \int_{\Re(u) =-\varepsilon-\Re(s)} 
	\frac{\Beta(-u, s+u)}{-u-s}
	(4n)^{u+s}
	\, du
	\\
	=
	4^{-s}\big(\psi(s)-\gamma+\log(4n)\big)
	-
	n^s (4n)^{-s}
	\int_{0}^{4n}
	\frac{1}{2\pi i} \int_{\Re(u) =-\varepsilon} 
	\Beta(-u, s+u)
	t^{u+s-1}
	\, du \, dt
	\\
	=
	4^{-s}\big(\psi(s)-\gamma+\log(4n)\big)
	-
	4^{-s}
	\int_{0}^{4n}
	(t+1)^{-s} t^{s-1}
	\, dt,
\end{multline*}
by \cite[6.422(3)]{GR}, for $\Re(s) > 0$. 
On the other hand,
$$
	2\int_0^1 \frac{f(y^{-1})}{y}\,dy
	=
	2n^s \int_0^1 (1+4ny^2)^{-s} y^{2s-1}\,dy
	=
	4^{-s}
	\int_{0}^{4n}
	(t+1)^{-s} t^{s-1}
	\,dt,
$$
and thus
$$
	\int_0^\infty \frac{f(y)+f(y^{-1})-f(0)}{y} \, dy
	=
	4^{-s}\big(\psi(s)-\gamma+\log(4n)\big)
	.
$$

It remains only to prove that the integral 
$$
	F(s) = \frac{1}{4\pi} \int_\R \Beta(s-ir, s+ir) \frac{\sigma_{2ir}(n)}{n^{ir}}
	\frac{\phi'}{\phi}\!\left(\frac{1}{2}+ir\right) dr
$$
has meromorphic continuation to $s\in \C$. 
Clearly $F(s)$ is analytic for $\Re(s)>0$. 
To get meromorphic continuation to $\Re(s) \leq 0$, we put $u=ir$
and then deform the contour around $u=0$:
$$
	F(s) 
	= 
	\frac{1}{4\pi i} \int_{C_0} \Beta(s-u, s+u) \frac{\sigma_{2u}(n)}{n^{u}} 
	\left(-\frac{{\zeta^*}'}{\zeta^*}(2u) + \frac{{\zeta^*}'}{\zeta^*}(1+2u)\right)
	du, 
$$
where 
$
	C_0 = \left\{u=it:|t|\geq \frac{1}{2}\right\} \cup \left\{ u\in \C:|u|=\frac{1}{2}, \; \Re(u)\geq 0\right\}
$.
Now we replace $u$ by $-u$ in the half of the integral containing
$\frac{{\zeta^*}'}{\zeta^*}(2u)$ and move the contour back to $C_0$: 
$$
	F(s) 
	= 
	-\frac{1}{4} \Beta(s, s) \sigma_0(n) 
	+
	\frac{1}{2\pi i} \int_{C_0} \Beta(s-u, s+u) \frac{\sigma_{2u}(n)}{n^{u}} 
	\frac{{\zeta^*}'}{\zeta^*}(1+2u)
	\, du. 
$$
Now let $M\in\Z_{\ge0}$, replace $u$ by $u+s$, and shift the contour to
$\Re(u)=M-\frac12$. Then we have
\begin{multline*}
	F(s)
	=
	\sum_{m=0}^{M-1} 
	\frac{(-1)^m}{m!} 
	\frac{\Gamma(2s+m)}{\Gamma(2s)}
	\frac{\sigma_{2s+2m}(n)}{n^{s+m}} \frac{{\zeta^*}'}{\zeta^*}(1+2s+2m)
	\\
	-\frac{1}{4} B(s, s) \sigma_0(n) 
	+
	\frac{1}{2\pi i} \int_{\Re(u)=M-\frac12} \Beta(-u, 2s+u) \frac{\sigma_{2s+2u}(n)}{n^{s+u}} 
	\frac{{\zeta^*}'}{\zeta^*}(1+2s+2u)
	\, du,
\end{multline*}
and this last line continues meromorphically to $\Re(s)>\frac14-\frac{M}2$. Taking $M$
arbitarily large, we conclude the meromorphic continuation of $F(s)$ to $\C$.

\end{proof}

\begin{proposition}\label{prop:specialization_N>1}
Let $N>1$ be a squarefree integer, $n\in\Z_{>0}$ with $(n,N)=1$
and $s\in\C$ with $\Re(s)>\frac12$.
Then
\begin{multline}\label{eqn:levelNplus}
	\sum_{\substack{\lambda\in\R_{>0}\\\lambda=\tfrac14+r^2}}
	\Tr T_n|_{\Anew(N, \lambda)} \Beta(s+ir, s-ir)
	+
	\mu(N) \frac{\sigma_1(n)}{\sqrt{n}} \Beta\!\left(s-\frac12, s+\frac12\right)
	\\
	+
	2\Lambda(N)
	\frac{1}{2\pi}\int_{-\infty}^\infty \Beta(s-ir, s+ir)
	\sigma_{-2ir}(n) n^{ir}
	\left(1-N^{-1-2ir}\right)^{-1}
	\,dr
	\\
	=
	4^{-s}
	\sum_{\substack{t\in\Z\\D=t^2-4n}} c_N^\circ(D) 
	\left\{\begin{array}{ll}
	\Phi\bigl(\frac{t^2}{4n}, s\bigr) & \text{if } D\neq 0,\\
	\frac{1}{2} 
	\sqrt{\frac{\pi}{n}}
	\frac{\Gamma(s+\frac12)}{\Gamma(s)} & \text{if } D=0
	\end{array}\right.
\end{multline}
and
\begin{multline}\label{eqn:levelNminus}
\sum_{\substack{\lambda\in\R_{>0}\\\lambda=\tfrac14+r^2}}
\Tr T_{-n}|_{\Anew(N, \lambda)} \Beta(s+ir, s-ir)
+
\mu(N)\frac{\sigma_1(n)}{\sqrt{n}} \Beta\!\left(s-\frac12,s+\frac12\right)
	\\
	+
	2\Lambda(N)
	\frac{1}{2\pi}\int_{-\infty}^\infty \Beta(s-ir, s+ir)
	\sigma_{-2ir}(n) n^{ir}
	\left(1-N^{-1-2ir}\right)^{-1}
	\,dr
	\\
	=
	n^s
	\sum_{\substack{t\in\Z\\D=t^2+4n}}
	\frac{c_N^\circ (D)}{D^s}
	, 
\end{multline}
where $c_N^\circ$ is as defined in \eqref{eqn:cN0def}. 
Both \eqref{eqn:levelNplus} and \eqref{eqn:levelNminus} continue
to meromorphic functions on $\C$ and are holomorphic for
$\Re(s)>\frac7{64}$, apart from simple poles of residue
$\mu(N)\frac{\sigma_1(n)}{\sqrt{n}}$ at $s=\tfrac12$.
\end{proposition}
\begin{proof}
With $q(v) = [4(1+v)]^{-s}$, 
as in the proof of Proposition~\ref{prop:specialization_N=1}, 
we have
$h(r) = \Beta(s+ir, s-ir)$.
Hence, by \eqref{eqn:tf2}, the left-hand sides of \eqref{eqn:levelNminus}
and \eqref{eqn:levelNplus} are 
$$
	\sum_{\lambda\in\R_{>0}}\Tr T_{\pm n}|_{\Anew(N, \lambda)}
	h\Bigl(\sqrt{\lambda-\tfrac14}\Bigr)
	=
	\sum_{\substack{\lambda\in\R_{>0}\\\lambda=\tfrac14+r^2}}
	\Tr T_n|_{\Anew(N, \lambda)} \Beta(s+ir, s-ir)
	.
$$

Let us first consider \eqref{eqn:levelNminus}. 
Then $f(y) = n^s |y|^{-2s}$, and $D=t^2+4n>0$, so that
$$
	f(\sqrt{D}) = n^s D^{-s}
	.
$$

Now we consider \eqref{eqn:levelNplus}, in which case
$f(y) = n^2(y^2+4n)^{-s}$ and $D=t^2-4n$. 
For $D=0$, 
$$
	\int_\R \frac{f(0)-f(y)}{y^2} \, dy
	=
	\frac{4^{-s} \sqrt{\pi}}{\sqrt{4n}} \frac{\Gamma(s+\frac12)}{\Gamma(s)}
	, 
$$
by \eqref{e:f0fy}. 
For $D<0$, 
$$
	\frac{\sqrt{|D|}}{\pi} \int_{\R} \frac{f(y)}{|D|+y^2} \, dy
	=
	4^{-s} I_{\frac{t^2}{4n}}\!\left(s, \frac{1}{2}\right)
	\left(\frac{t^2}{4n}\right)^{-s}, 
$$
by \eqref{e:fy_|D|}. 

For $D=t^2\pm 4n$ with $n\geq 1$, 
$$
	f\!\left(aN^r - \frac{n}{aN^r}\right)
	=
	n^s\left(aN^r+\frac{n}{aN^r}\right)^{-2s}.
$$
By \cite[6.422(1)]{GR}, for $1-\Re(s) < \sigma< \Re(s)$, we have
$$
	n^s\left(aN^r+\frac{n}{aN^r}\right)^{-2s}
	=
	\frac{1}{2\pi i}\int_{\Re(u) = \sigma} \Beta(s-u, s+u)
	\left(\frac{n}{a^2N^{2r}}\right)^u \,du
	,
$$
so that
$$
	\sum_{\substack{a\in\Z_{>0}\\a\mid n}} \sum_{r=0}^\infty 
	N^{-r} 
	f\!\left(aN^r-\frac{n}{aN^r}\right)
	=
	\frac{1}{2\pi i}\int_{\Re(u) = \sigma} \Beta(s-u, s+u)
	\sigma_{-2u}(n) n^{u}
	\left(1-N^{-1-2u}\right)^{-1}
	\,du
	.
$$
This integral and the sum over $\lambda$ have meromorphic continuation to $s\in \C$,
by similar arguments to those of Proposition~\ref{prop:specialization_N=1}.
Finally, the fact that the sum over $\lambda$ is holomorphic for
$\Re(s)>\frac7{64}$ follows from the best known bound towards the
Selberg eigenvalue conjecture, due to Kim and Sarnak \cite{ks}.
\end{proof}

\section{Proofs of the main results}\label{sec:applications}
\subsection{Proof of Theorem~\ref{thm:main1}}

By \eqref{eqn:level1minus} in Proposition ~\ref{prop:specialization_N=1}, 
the series given in (1) can be written as 
$$
	\sum_{\substack{t\in\Z\\\sqrt{t^2+4n}\notin\Z}}
	\frac{L(1, \psi_{t^2+4n})}{(t^2+4n)^s}
	=
	n^{-s} \left(F_1(s) + F_2(s)+F_3(s)\right)
	, 
$$
where
$$
	F_1(s) 
	= 
	\sum_{\substack{\lambda\in\R_{>0}\\\lambda=\frac{1}{4}+r^2}}
	\Tr T_{-n}|_{\A(1, \lambda)} \Beta(s+ir, s-ir)
	+
	\frac{\sigma_1(n)}{\sqrt{n}} \Beta\!\left(s-\frac{1}{2}, s+\frac{1}{2}\right)
	, 
$$
$$
	F_2(s)
	=
	-
	\frac{1}{4\pi} \int_{\R} \Beta(s-ir, s+ir)\frac{\sigma_{2ir}(n)}{n^{ir}} 
	\frac{\phi'}{\phi}\!\left(\frac{1}{2}+ir\right) dr
	- \frac{\sigma_0(n)}{4} \Beta(s, s)
$$
and 
$$
	F_3(s) 
	=
	-\sum_{\substack{t\in\Z\\t^2+4n=\ell^2, \ell\in \Z_{>0}}}
	\frac{\sum_{m\mid \ell} \Lambda(m) (1-m^{-1})
	+\frac12\left(\psi\!\left(s+\frac{1}{2}\right)-\psi\!\left(s\right)\right)}
	{\ell^{2s}}
	. 
$$

The meromorphic continuation of $F_1(s)$ and $F_2(s)$ was shown in the
proof of Proposition~\ref{prop:specialization_N=1}. In particular,
$F_1(s)$ has simple poles at $s=-m+\frac12$ and $s=-m\pm ir$ for
$m\in\Z_{\ge0}$, while $F_2(s)$ has simple poles at $s=-m+\frac{\rho-1}2$
for $m\in\Z_{\ge0}$ and $\rho$ a zero or pole of $\zeta^*(s)$.  Finally,
the series $F_3(s)$ has finitely many terms and is entire apart from
simple poles for $s\in-\frac12\Z_{\geq 0}$.  Moreover, one can check
that the poles of $F_2(s)$ and $F_3(s)$ at $s=0$ cancel out, so the only
poles of \eqref{eqn:level1series} for $\Re(s)\ge0$ are at $s=\frac12$
and $s=\pm ir$, with residues $\sigma_{-1}(n)$ and $n^{\mp ir}\Tr
T_{-n}|_{\A(1,\frac14+r^2)}$, respectively.  This proves (1).

Let $N$ be a prime and $n$ be a positive integer such that $\left(\frac{-4n}{N}\right)=-1$. 
Then $N\nmid t^2+4n = d\ell^2=D$ and 
$$
	c_N^\circ(D) = \left(\psi_D(N)-1\right) L(1, \psi_D)
	= \left(\left(\frac{t^2+4n}{N}\right)-1\right) L(1, \psi_{t^2+4n}). 
$$
Combining \eqref{eqn:levelNminus} in Proposition \ref{prop:specialization_N>1}
 and \eqref{eqn:level1minus} in Proposition \ref{prop:specialization_N=1}, 
we see that the series
$$
	\sum_{\substack{t\in\Z\\\sqrt{t^2+4n}\notin\Z}}
	\frac{c_N^\circ(t^2+4n)}{(t^2+4n)^s}
	+
	\sum_{\substack{t\in\Z\\\sqrt{t^2+4n}\notin\Z}}
	\frac{L(1, \psi_{t^2+4n})}{(t^2+4n)^s}
	=
	\sum_{\substack{t\in\Z\\\sqrt{t^2+4n}\notin\Z}}
	\frac{L(1, \psi_{t^2+4n}) \left(\frac{t^2+4n}{N}\right)}
	{(t^2+4n)^s}
$$
has meromorphic continuation to $\C$. This proves (2). 

Finally, we turn to (3). If $N=2$ then there are no primes
$n$ satisfying $\left(\frac{-4n}{N}\right)=-1$. However, the Selberg
eigenvalue conjecture is true for $\Gamma_0(2)$ \cite{huxley}, so (3)
is vacuously true in this case. Henceforth we will assume that $N>2$.

If the Selberg conjecture holds for $\Gamma_0(N)$ then, since it also
holds for $\Gamma_0(1)$,
the first terms of
both \eqref{eqn:level1minus} and \eqref{eqn:levelNminus} are holomorphic
for $\Re(s) > 0$, and the second terms cancel.
In the other direction, let $\{f_j\}_{j=1}^\infty$ be a complete
sequence of arithmetically normalized Hecke--Maass newforms on
$\Gamma_0(N)\backslash\HH$, with parities $\epsilon_j\in\{0,1\}$,
Laplace eigenvalues $\frac14+r_j^2$ and Hecke eigenvalues $\lambda_j(n)$.
We need the following lemma.
\begin{lemma}\label{lem:nonvanishing}
Let $J$ be a finite set of positive integers and let $c_j\in \C^\times$ be given for each $j\in J$. 
Then there is a prime number $n$ such that
$\left(\frac{-4n}{N}\right)=-1$ and
$$
\sum_{j\in J}c_j\lambda_j(n)\ne0.
$$
\end{lemma}
\begin{proof}
The main tool is the Rankin--Selberg method, from which it follows
that if $f$ and $g$ are normalized Hecke--Maass newforms (of
possibly different levels) with Fourier coefficients $\lambda_f(n)$
and $\lambda_g(n)$, respectively, then
\begin{equation}\label{eqn:rs}
	\lim_{x\to\infty}\frac{1}{\pi(x)}\sum_{p\le x}
	\lambda_f(p)\overline{\lambda_g(p)}
	=
	\begin{cases}
	1&\mbox{if }f=g,\\
	0&\mbox{if }f\ne g
	,
	\end{cases}
\end{equation}
where the sum runs through all prime numbers $p\le x$.

For any prime $p$, put $S_p=\sum_{j\in J}c_j\lambda_j(p)$.
For large $x\ge N$, we consider the sum
$$
	\begin{aligned}
	\sum_{\substack{p\le x\\\left(\frac{-4p}{N}\right)=-1}}
	|S_p|^2
	&=\frac12\sum_{p\le x}|S_p|^2 -\frac12\left(\frac{-4}{N}\right)
	\sum_{p\le x}\left(\frac{p}{N}\right)|S_p|^2
	-\frac12 |S_N|^2.
	\end{aligned}
$$
Opening up the first sum, we have
$$
	\sum_{p\le x}|S_p|^2
	=
	\sum_{j,k\in J}c_j\overline{c_k}
	\sum_{p\le x}\lambda_j(p)\overline{\lambda_k(p)}.
$$
By \eqref{eqn:rs}, the inner sum is $o(\pi(x))$ if $j\ne k$ and
$(1+o(1))\pi(x)$ otherwise; thus, in total, the first sum is
$\left(\sum_{j\in J}|c_j|^2+o(1)\right)\pi(x)$.

Expanding the second sum in the same way, we obtain
$$
	\sum_{j,k\in J}c_j\overline{c_k}
	\sum_{p\le x}\left(\frac{p}{N}\right)\lambda_j(p)\overline{\lambda_k(p)}.
$$
Since $f_j$ is a newform of prime level $N$, 
$\left(\frac{p}{N}\right)\lambda_j(p)$ is the $p$th Fourier coefficient
of a newform of level $N^2$ (which is therefore distinct from $f_k$ for
every $k$). Thus, \eqref{eqn:rs} implies that the
second sum is $o(\pi(x))$.

Finally $\frac{1}{2}|S_N|^2$ is bounded independent of $x$. 
Putting these together, we have
$$
	\sum_{\substack{p\le x\\\left(\frac{-4p}{N}\right)=-1}}|S_p|^2
	=
	\left(\frac12\sum_{j\in J} |c_j|^2+o(1)\right)\pi(x)
	\quad\mbox{as }x\to\infty.
$$
Since $c_j\neq 0$, the last line is
positive for large enough $x$,
and thus there exists $p$ satisfying
$\left(\frac{-4p}{N}\right)=-1$ and $S_p\ne 0$, as required.
\end{proof}

To conclude the proof, suppose that the Selberg conjecture is false for
$\Gamma_0(N)$.  Let $\lambda_\text{min}\in(0,\frac14)$ be the smallest
non-zero eigenvalue, and put
$J=\{j\in\Z_{>0}:\frac14+r_j^2=\lambda_\text{min}\}$,
$c_j=(-1)^{\epsilon_j}$.
Then Lemma~\ref{lem:nonvanishing} implies that
there exists a prime $n$ such that $\left(\frac{-4n}{N}\right)=-1$ and
$$
	0\ne\sum_{j\in J}(-1)^{\epsilon_j}\lambda_j(n)
	=
	\Tr T_{-n}|_{\A(N,\lambda_\text{min})}.
$$
Thus, for this choice of $n$, \eqref{eqn:quadtwist} has a pole
at $s=\sqrt{\frac14-\lambda_\text{min}}>0$.
\qed

\begin{remark}
Using effective estimates in the Rankin--Selberg method, one could give
an upper bound for the $n$ produced by Lemma~\ref{lem:nonvanishing}, so
part (3) of Theorem~\ref{thm:main1} could be strengthened to an
equivalence with finitely many series. Alternatively, using known
results from functoriality would enable equivalences with thinner
infinite sequences
of $n$; for instance, if $N\equiv3\pmod*{4}$, then using functoriality
of the $k$th symmetric powers for $k\le 4$, we may take $n$ to be the
fourth power of a prime.
\end{remark}

\subsection{Proof of Theorem~\ref{thm:main2}}
Define $E_N(s)=\prod_{p\mid N}(1-p^{-s})$.
We multiply \eqref{eqn:levelNminus} by $n^{-s} \cdot \Gamma_\R(2s)\zeta_N^*(4s)/E_N(2s+1)$ and sum
over square values of $n$ co-prime to $N$. The result can be expressed
in the form $L_1+L_2+L_3=R$, where
\begin{align*}
L_1&=\mu(N)\frac{\Gamma_\R(2s)\zeta_N^*(4s)}{E_N(2s+1)}
\frac{\Gamma_\R(2s-1)\Gamma_\R(2s+1)}{\Gamma_\R(4s)}
\sum_{\substack{n\in\Z_{>0}\\(n,N)=1}}\frac{\sigma_1(n^2)}{n^{2s+1}},\\
L_2&=\frac{\Gamma_\R(2s)\zeta_N^*(4s)}{E_N(2s+1)}
\sum_{j=1}^\infty(-1)^{\epsilon_j}
\frac{\Gamma_\R(2s-2ir_j)\Gamma_\R(2s+2ir_j)}{\Gamma_\R(4s)}
\sum_{\substack{n\in\Z_{>0}\\(n,N)=1}}
\frac{\lambda_j(n^2)}{n^{2s}},\\
L_3&=\frac{\Gamma_\R(2s) \zeta_N^*(4s)}{E_N(2s+1)} 2\Lambda(N) \frac{1}{2\pi} \int_{-\infty}^\infty \Beta (s-ir, s+ir) 
	\sum_{\substack{n\in\Z_{>0}\\(n, N)=1}} \frac{\sigma_{-2ir}(n^2)(n^2)^{ir}}{n^{2s}} 
	\left(1-N^{-1-2ir}\right)^{-1} \, dr,
\end{align*}
and
$$
	R 
	= 
	\frac{\Gamma_\R(2s)\zeta_N^*(4s)}{E_N(2s+1)}
	\sum_{\substack{t\in\Z, n\in\Z_{>0}\\(n, N)=1}}
	\frac{c_N^\circ (t^2+4n^2)}{(t^2+4n^2)^s}
	.
$$
By Atkin--Lehner theory, for any $p\mid N$,
we have $\lambda_j(p^{2k})=p^{-k}$. Thus,
$$
\sum_{\substack{n\in\Z_{>0}\\(n,N)=1}}\frac{\lambda_j(n^2)}{n^{2s}}
=E_N(2s+1)\sum_{n=1}^\infty\frac{\lambda_j(n^2)}{n^{2s}},
$$
so that
$$
L_2=\sum_{j=1}^\infty(-1)^{\epsilon_j}L^*(2s,\Sym^2{f_j}).
$$
Similarly,
$$
\sum_{\substack{n\in\Z_{>0}\\(n,N)=1}}\frac{\sigma_1(n^2)}{n^{2s+1}}
=\frac{E_N(2s)E_N(2s+1)E_N(2s-1)}{E_N(4s)}
\frac{\zeta(2s)\zeta(2s+1)\zeta(2s-1)}{\zeta(4s)},
$$
so that
\begin{align*}
L_1&=\mu(N)N^{2s}E_N(2s)E_N(2s-1)
\zeta^*(2s)\zeta^*(2s+1)\zeta^*(2s-1)\\
&=\sqrt{N}\zeta^*(2s)\zeta_N^*(-2s)\zeta_N^*(2s-1).
\end{align*}

Turning to the right-hand side, we define
$$
r^{(M)}(D)=\tfrac12\#\bigl\{(x,y)\in\Z^2:D=x^2+4y^2,\;(y,M)=1\bigr\},
$$
for $M\in\Z_{>0}$, so that
$$
	R=\frac{\Gamma_\R(2s)\zeta_N^*(4s)}{E_N(2s+1)}
\sum_{D=1}^\infty\frac{c_N^\circ(D)r^{(N)}(D)}{D^s}.
$$
Note that $c_N^\circ(m^2D)=m^{-1}c_N^\circ(D)$ for any $m\mid N^\infty$
and $D\in\DD$.  Hence, expanding the factor of $E_N(2s+1)^{-1}$
and writing $r^{(N)}(x)=0$ if $x$ is not a positive integer,
we get
\begin{equation}\label{eqn:R1R2first}
\begin{aligned}
R&=\Gamma_\R(2s)\zeta_N^*(4s)
\sum_{m\mid N^\infty}m^{-2s-1}
\sum_{D=1}^\infty\frac{c_N^\circ(D)r^{(N)}(D)}{D^s}\\
&=\Gamma_\R(2s)\zeta_N^*(4s)
\sum_{m\mid N^\infty}
\sum_{D=1}^\infty\frac{c_N^\circ(m^2D)r^{(N)}(D)}{(m^2D)^s}\\
&=\Gamma_\R(2s)\zeta_N^*(4s)
\sum_{m\mid N^\infty}
\sum_{D=1}^\infty\frac{c_N^\circ(D)r^{(N)}(Dm^{-2})}{D^s}\\
&=\Gamma_\R(2s)\zeta_N^*(4s)
\sum_{D=1}^\infty\frac{c_N^\circ(D)}{D^s}
\sum_{m\mid N^\infty}r^{(N)}(Dm^{-2}).
\end{aligned}
\end{equation}
The inner sum is evaluated by the following three lemmas.
\begin{lemma}\label{lem:rNell2}
For positive integers $M,\ell$,
$$
r^{(M)}(\ell^2)=\#\bigl\{(a,n)\in\Z_{>0}^2:a\mid n^2,
\;(n,M)=1,\;a+\tfrac{n^2}a=\ell\bigr\}+
\begin{cases}
1&\mbox{if }M=1,\\
0&\mbox{if }M>1.
\end{cases}
$$
\end{lemma}
\begin{proof}
Put $\delta_M=1$ if $M=1$ and $\delta_M=0$ otherwise. Then we have
\begin{align*}
r^{(M)}(\ell^2)
&=\tfrac12\#\bigl\{(x,y)\in\Z^2:\ell^2=x^2+4y^2,\;(y,M)=1\bigr\}\\
&=\delta_M+\#\{(x,n)\in\Z\times\Z_{>0}:\ell^2=x^2+4n^2,\;(n,M)=1\}.
\end{align*}
Note that if $\ell^2=x^2+4n^2$ with $n>0$ then $x\equiv\ell\pmod*{2}$
and $|x|<\ell$, so this becomes
\begin{align*}
&\delta_M+\#\bigl\{(x,n)\in\Z\times\Z_{>0}:
\tfrac{\ell+x}2\cdot\tfrac{\ell-x}2=n^2,
\;(n,M)=1,\;x\equiv\ell\pmod*{2},|x|<\ell\bigr\}\\
&=\delta_M+\#\bigl\{(a,n)\in\Z_{>0}^2:a(\ell-a)=n^2,\;(n,M)=1\bigr\}\\
&=\delta_M+\#\bigl\{(a,n)\in\Z_{>0}^2:a\mid n^2,
\;(n,M)=1,\;a+\tfrac{n^2}a=\ell\bigr\}.
\end{align*}
\end{proof}
\begin{lemma}\label{lem:rNtelescope}
Let $M,D$ be positive integers with $M$ squarefree,
and let $p$ be a prime divisor of $M$. Then
$$
\sum_{k=0}^\infty r^{(M)}(Dp^{-2k})=r^{(M/p)}(D)
-\tfrac12\#\bigl\{(x,y)\in\Z^2:
Dp^{-2\lfloor\ord_p(D)/2\rfloor}=x^2+4y^2,\;p\mid y\bigr\}.
$$
In particular, if $\psi_d(p)\ne1$,
where $d$ denotes the discriminant of $\Q(\sqrt{D})$, then
$$
\sum_{k=0}^\infty r^{(M)}(Dp^{-2k})=r^{(M/p)}(D).
$$
\end{lemma}
\begin{proof}
If $p^2\mid n$ then
\begin{align*}
r^{(M)}(n)&=\tfrac12\#\bigl\{(x,y)\in\Z^2:n=x^2+4y^2,\;(y,M)=1\bigr\}\\
&=r^{(M/p)}(n)
-\tfrac12\#\bigl\{(x,y)\in\Z^2:n=x^2+4y^2,\;(y,M/p)=1,\;p\mid y\bigr\}\\
&=r^{(M/p)}(n)-r^{(M/p)}(np^{-2}).
\end{align*}
Now put $e=\lfloor{\ord_p(D)/2}\rfloor$. Then for $k<e$ we
may apply the above with $n=Dp^{-2k}$:
$$
r^{(M)}(Dp^{-2k})=r^{(M/p)}(Dp^{-2k})-r^{(M/p)}(Dp^{-2(k+1)}).
$$
Hence we have the telescoping sum
$$
\sum_{k=0}^\infty r^{(M)}(Dp^{-2k})
=\sum_{k=0}^e r^{(M)}(Dp^{-2k})=r^{(M/p)}(D)-r^{(M/p)}(Dp^{-2e})
+r^{(M)}(Dp^{-2e}).
$$
The first conclusion follows on noting that
$$
r^{(M/p)}(Dp^{-2e})-r^{(M)}(Dp^{-2e})
=\tfrac12\#\bigl\{(x,y)\in\Z^2:Dp^{-2e}=x^2+4y^2,\;p\mid y\bigr\}.
$$

As for the second, if $\ord_p(D)$ is odd then $\ord_p(Dp^{-2e})=1$.
On the other hand, if $p\mid y$ then $x^2+4y^2$ is either invertible
mod $p$ or divisible by $p^2$. Hence
$$
\tfrac12\#\bigl\{(x,y)\in\Z^2:Dp^{-2e}=x^2+4y^2,\;p\mid y\bigr\}=0
$$
in this case.

Suppose now that $\ord_p(D)$ is even and $\psi_d(p)\ne1$.
If $p$ is odd then it follows that
$\left(\frac{Dp^{-2e}}{p}\right)=-1$, so again
$Dp^{-2e}=x^2+4y^2$ is not solvable with $p\mid y$.
For $p=2$, we distinguish between even and odd values of $d$.
If $d$ is even then we have
$Dp^{-2e}=\frac14dm^2$, where $m$ is odd;
hence, the equation $Dp^{-2e}=x^2+4y^2$ is not solvable, since
$\frac14dm^2\equiv2$ or $3\pmod*{4}$, while
$x^2+4y^2\equiv0$ or $1\pmod*{4}$.
If $d$ is odd then $d\equiv5\pmod*{8}$ since $\psi_d(p)\ne1$,
so $Dp^{-2e}=dm^2\equiv5\pmod*{8}$. Again we find that
$Dp^{-2e}=x^2+4y^2$ is not solvable with $y$ even.
\end{proof}
\begin{corollary}\label{cor:rNtelescope}
For any positive integer $D$,
$$
c_N^\circ(D)\sum_{m\mid N^\infty}r^{(N)}(Dm^{-2})
=c_N^\circ(D)r(D)-\delta_{N,D},
$$
where
$$
\delta_{N,D}=\begin{cases}
\frac{\Lambda(N)}{(\ell,N^\infty)}\bigl(1+
\#\{(a,n)\in\Z_{>0}^2:a\mid n^2,\;N\mid n,\;a+\frac{n^2}{a}
=\frac{\ell}{(\ell,N^\infty)}\}\bigr)
&\mbox{if }D=\ell^2,\\
0&\mbox{otherwise}.
\end{cases}
$$
\end{corollary}
\begin{proof}
If $D\equiv2$ or $3\pmod*{4}$ then both sides vanish. If $D=d\ell^2\in\DD$
with $d\ne1$ then $c_N^\circ(D)=0$ unless $\psi_d(p)\ne1$ for every
$p\mid N$; in that case, we may apply the second conclusion of
Lemma~\ref{lem:rNtelescope} inductively to see that
$\sum_{m\mid N^\infty}r^{(N)}(Dm^{-2})=r(D)$.
Finally, if $D=\ell^2$ is a square then again both sides are $0$ unless $N$ is
prime, and in that case we apply the first conclusion of
Lemma~\ref{lem:rNtelescope} with $M=p=N$. The stated formula for
$\delta_{N,D}$ follows from Lemma~\ref{lem:rNell2}, since
\begin{align*}
\tfrac12&\#\bigl\{(x,y)\in\Z^2:\ell^2N^{-2\ord_N(\ell)}=x^2+4y^2,
\;N\mid y\bigr\}\\
&=r^{(1)}(\ell^2N^{-2\ord_N(\ell)})-r^{(N)}(\ell^2N^{-2\ord_N(\ell)})\\
&=1+\#\bigl\{(a,n)\in\Z_{>0}^2:a\mid n^2,\;N\mid n,\;a+\tfrac{n^2}{a}
=\tfrac{\ell}{(\ell,N^\infty)}\bigr\}.
\end{align*}
\end{proof}

We apply this to \eqref{eqn:R1R2first}. When $N$ is composite we have
$c_N^\circ=c_N$, so that $R=\sum_{D=1}^\infty c_N(D)r(D)D^{-s}$,
which completes the proof in that case.  Henceforth we assume that
$N$ is prime; in particular, $E_N(s)=1-N^{-s}$. Thus, by
Corollary~\ref{cor:rNtelescope},
\begin{equation}\label{eqn:R1R2second}
\begin{aligned}
&R-\Gamma_\R(2s)\zeta_N^*(4s)
\sum_{D=1}^\infty\frac{c_N^\circ(D)r(D)}{D^s}\\
&=-\Lambda(N)\Gamma_\R(2s)\zeta_N^*(4s)\sum_{\ell=1}^\infty
\frac{1+\#\{(a,n)\in\Z_{>0}^2:a\mid n^2,\;N\mid n,\;a+\tfrac{n^2}{a}
=\tfrac{\ell}{(\ell,N^\infty)}\}}
{(\ell,N^\infty)\ell^{2s}}\\
&=-\Lambda(N)\Gamma_\R(2s)\zeta_N^*(4s)\sum_{k=0}^\infty N^{-k(2s+1)}
\sum_{\substack{\ell_0\in\Z_{>0}\\(\ell_0,N)=1}}
\frac{1+\#\{(a,n)\in\Z_{>0}^2:
a\mid n^2,\;N\mid n,\;a+\tfrac{n^2}{a}=\ell_0\}}
{\ell_0^{2s}}\\
&=-\frac{\Lambda(N)\Gamma_\R(2s)\zeta_N^*(4s)}{E_N(2s+1)}
\Biggl(E_N(2s)\zeta(2s)
+\sum_{\substack{n\in\Z_{>0}\\N\mid n}}
\sum_{\substack{a\mid n^2\\(a+n^2/a,N)=1}}
\left(a+\frac{n^2}{a}\right)^{-2s}
\Biggr).
\end{aligned}
\end{equation}
We write $n=n_0N^r$ for $r>0$ and $(n_0,N)=1$.
Then the condition $(a+n^2/a,N)=1$ is satisfied if and only if
$\ord_N(a)\in\{0,2r\}$, so we have
\begin{align*}
\sum_{\substack{n\in\Z_{>0}\\N\mid n}}
&\sum_{\substack{a\mid n^2\\(a+n^2/a,N)=1}}
\left(a+\frac{n^2}{a}\right)^{-2s}
=\sum_{r=1}^\infty\sum_{\substack{n_0\in\Z_{>0}\\(n_0,N)=1}}
\sum_{a\mid n_0^2}\Biggl(
\left(a+\frac{n_0^2N^{2r}}{a}\right)^{-2s}
+\left(aN^{2r}+\frac{n_0^2}{a}\right)^{-2s}\Biggr)\\
&=2\sum_{r=1}^\infty\sum_{\substack{n_0\in\Z_{>0}\\(n_0,N)=1}}
\sum_{a\mid n_0^2}
\left(aN^{2r}+\frac{n_0^2}{a}\right)^{-2s}
=2\sum_{\substack{r\in\Z\\r<0}}\sum_{\substack{n_0\in\Z_{>0}\\(n_0,N)=1}}
N^{4rs}\sum_{a\mid n_0^2}
\left(a+\frac{n_0^2N^{2r}}{a}\right)^{-2s}.
\end{align*}
Substituting this into \eqref{eqn:R1R2second} and combining
with $R_3$, we obtain
\begin{equation}\label{eqn:R1R2R3}
\begin{aligned}
R-\Gamma_\R(2s)\zeta_N^*(4s)
\sum_{D=1}^\infty\frac{c_N^\circ(D)r(D)}{D^s}
&=
-\Lambda(N)\zeta_N^*(4s)\frac{E_N(2s)\zeta^*(2s)}{E_N(2s+1)}\\
&-\frac{2\Lambda(N)\Gamma_\R(2s)\zeta_N^*(4s)}{E_N(2s+1)}
\sum_{\substack{n\in\Z_{>0}\\(n,N)=1}}\sum_{a\mid n^2}
\sum_{\substack{r\in\Z, \\ r<0}}\left(a+\frac{n^2N^{2r}}{a}\right)^{-2s}
N^{4rs}
, 
\end{aligned}
\end{equation}
which we write as $R_1'+R_2'$. 

Next, we convert $R_2'$ into an integral by
applying \cite[6.422(3)]{GR}:
\begin{equation}\label{eqn:an2aint}
\left(a+\frac{n^2N^{2r}}{a}\right)^{-2s}
=\frac1{2\pi i}\int_{\Re(u)=0}
\frac{\Gamma_\R(2s+2u)\Gamma_\R(2s-2u)}{\Gamma_\R(4s)}
a^{-2u}(nN^r)^{2(u-s)}\,du.
\end{equation}
We have
\begin{align*}
\sum_{\substack{r\in\Z,\\ r<0}}N^{2r(u-s)} N^{4rs}
&=
\frac{N^{-2u-2s}}{1-N^{-2u-2s}}
=\frac{E_N(2s+1)}{E_N(1-2u)E_N(2u+2s)}
-\frac1{1-N^{2u-1}},
\end{align*}
and
$$
\sum_{\substack{n\in\Z_{>0}\\(n,N)=1}}
\sum_{a\mid n^2}a^{-2u}n^{2u-2s}
=\frac{E_N(2s)E_N(2s+2u)E_N(2s-2u)}{E_N(4s)}
\frac{\zeta(2s)\zeta(2s+2u)\zeta(2s-2u)}{\zeta(4s)},
$$
so that
$$
R_2'-L_3=-2\Lambda(N)N^{2s}E_N(2s)\zeta^*(2s)
\int_{\Re(u)=0}\frac{E_N(2s-2u)}{E_N(1-2u)}
\zeta^*(2s+2u)\zeta^*(2s-2u)\,\frac{du}{2\pi i}.
$$
Now, since
$$
\frac{E_N(2s-2u)}{E_N(1-2u)}=
N^{1-2s}\left(1-\frac{E_N(1-2s)}{E_N(1-2u)}\right),
$$
this equals
$$
-2N\Lambda(N)E_N(2s)\zeta^*(2s)
\bigl(A(s)-E_N(1-2s)B_N(s)\bigr),
$$
where
$$
A(s)=\frac1{2\pi i}\int_{\Re(u)=0}
\zeta^*(2s+2u)\zeta^*(2s-2u)\,du
$$
and
$$
B_N(s)=\frac1{2\pi i}\int_{\Re(u)=0}
\frac{\zeta^*(2s+2u)\zeta^*(2s-2u)}{E_N(1-2u)}\,du.
$$

Next, to compute the contribution from $c_N-c_N^\circ$, we have the
following:
\begin{lemma}
$$
\zeta^*(4s)\sum_{n=1}^\infty\frac{r(\ell^2)}{\ell^{2s}}
=\zeta(2s)(A(s)+\zeta^*(4s)).
$$
\end{lemma}
\begin{proof}
By Lemma~\ref{lem:rNell2} with
$M=1$, we have
$$
\sum_{\ell=1}^\infty\frac{r(\ell^2)}{\ell^{2s}}
=\zeta(2s)+\sum_{n=1}^\infty\sum_{a\mid n^2}
\left(a+\frac{n^2}{a}\right)^{-2s}.
$$
On the other hand, by \eqref{eqn:an2aint} with $r=0$,
\begin{align*}
\sum_{n=1}^\infty&\sum_{a\mid n^2}
\left(a+\frac{n^2}{a}\right)^{-2s}
=\frac1{2\pi i}\int_{\Re(u)=0}
\frac{\Gamma_\R(2s+2u)\Gamma_\R(2s-2u)}{\Gamma_\R(4s)}
\sum_{n=1}^\infty\sigma_{-2u}(n^2)n^{2(u-s)}\,du\\
&=\frac1{2\pi i}\int_{\Re(u)=0}
\frac{\Gamma_\R(2s+2u)\Gamma_\R(2s-2u)}{\Gamma_\R(4s)}
\frac{\zeta(2s)\zeta(2s+2u)\zeta(2s-2u)}{\zeta(4s)}\,du
=\frac{\zeta(2s)}{\zeta^*(4s)}A(s).
\end{align*}
\end{proof}

Comparing the definitions of $c_N$ and $c_N^\circ$, we thus have
\begin{multline}\label{eqn:ABrhs}
	R-L_3
	-\Gamma_\R(2s)\zeta_N^*(4s)\sum_{D=1}^\infty
\frac{c_N(D)r(D)}{D^s}
	\\
	=
	\Lambda(N)\zeta^*(2s)\biggl[
	2NA(s)\left(\frac{E_N^*(4s)}{N+1}-E_N(2s)\right)
	+
	2NB_N(s)E_N(2s)E_N(1-2s)
	\\
	\hspace{3cm}+\zeta_N^*(4s)
	\left(\frac{2N}{N+1}-\frac{E_N(2s)}{E_N(2s+1)}\right)
	\biggr].
\end{multline}
Note that, for prime $N$,
$$
\frac{E_N^*(4s)}{N+1}-E_N(2s)
=-\frac{N}{N+1}E_N(2s)E_N(1-2s)
$$
and
$$
\frac1{E_N(1-2u)}-\frac{N}{N+1}
=\frac1{N+1}\frac{E_N^*(4u)}{E_N(2u)E_N(1-2u)},
$$
so on making the substitution $u\mapsto\frac{u}2$,
\eqref{eqn:ABrhs} becomes
\begin{equation}\label{eqn:ABrhs2}
\frac{N\Lambda(N)}{N+1}\zeta^*(2s)\biggl[\IN(2s;0)
+\zeta_N^*(4s)\left(2-\frac{(N+1)E_N(2s)}{NE_N(2s+1)}\right)
\biggr].
\end{equation}

Shifting the contour of $\IN$ to $\Re(u)=-\sigma$, we pass poles at
$u=1-2s$ and $u=-2s$, with residues
$$
E_N^*(2-4s)\zeta^*(4s-1)=\zeta_N^*(2-4s)
$$
and
$$
-\frac{E_N^*(2s)E_N^*(1-2s)}{E_N^*(-2s)E_N^*(1+2s)}
E_N^*(-4s)\zeta^*(4s)
=\frac{E_N(2s-1)}{NE_N(2s+1)}\zeta_N^*(4s),
$$
respectively.
Finally, we have
$$
\frac{E_N(2s-1)}{NE_N(2s+1)}+2-\frac{(N+1)E_N(2s)}{NE_N(2s+1)}=1,
$$
so that \eqref{eqn:ABrhs2} is
$$
\frac{N\Lambda(N)}{N+1}\zeta^*(2s)\bigl(\IN(2s;-\sigma)
+\zeta_N^*(4s)+\zeta_N^*(2-4s)\bigr),
$$
as required.

The analytic continuation and functional equation of $L^*(2s, \Sym^2 f_j)$ were
proved by Gelbart and Jacquet \cite{GJ78} following ideas of Shimura \cite{Shi75}.
By Stirling's formula and the convexity bound, we have
$$
	L^*(2s, \Sym^2 f_j) \ll_{K, \varepsilon} e^{-(\pi-\varepsilon)|r_j|}, 
$$
uniformly for $s\in K$, for any compact subset $K\subset \C$ and $\varepsilon>0$.
Combining this with the Weyl-type estimate
$\#\{j:|r_j|\leq T\} \ll T^2$, we see that
the series $L_2$ converges absolutely to an entire function of $s$.
Similarly, $\IN(2s; \sigma)$ converges absolutely for all $s$ with $\Re(2s)<\sigma$
and satisfies a functional equation.
\qed

\subsection{Proof of Theorem~\ref{thm:main3}}
Let 
$$
	F(s) 
	= 
	\sum_{\substack{0<D\in\DD\\\sqrt{D}\notin\Z}}\frac{L(1, \psi_D) r(D)}{D^s}. 
$$
It is straightforward to show that
$L(1,\psi_D)r(D)\ll_\varepsilon D^\varepsilon$, for
any $D\in\DD$ with $D>0$ and $\sqrt{D}\notin\Z$.
Using this estimate in \cite[\S{}II.1, Cor.~2.1]{tenenbaum},
for any $X\ge T\ge2$, we have
$$
\sum_{\substack{0<D\le X\\\sqrt{D}\notin\Z}}
L(1,\psi_D)r(D)=\frac1{2\pi i}\int_{\kappa-iT}^{\kappa+iT}
F(s)X^s\frac{ds}{s}
+O_\varepsilon\!\left(\frac{X^{1+\varepsilon}}{T}\right),
$$
where $\kappa=1+1/\log{X}$.

Our goal now is to shift the contour to $\Re(s)=\frac14$.
To that end, we will show
that $F(s)$ is analytic on $\{s\in\C:\Re(s)\ge\frac14\}$, apart from
poles at $s\in\{\frac12,1\}$, and satisfies the estimates
\begin{equation}\label{e:F14}
F(\tfrac14+it)\ll_\varepsilon(1+|t|)^{\frac94+\varepsilon}
\quad\mbox{and}\quad
\int_{-T}^T|F(\tfrac14+it)|\,dt
\ll_\varepsilon T^{\frac{11}4+\varepsilon}.
\end{equation}
Hence, by partial integration, we have
$$
\int_{\frac14-iT}^{\frac14+iT}F(s)X^s\frac{ds}{s}
\ll_\varepsilon X^{\frac14}T^{\frac74+\varepsilon}.
$$
By convexity, it follows from the first estimate in \eqref{e:F14} that
$$
F(\sigma\pm iT)\ll_\varepsilon T^{3(1-\sigma)+\varepsilon}
\quad\mbox{for }\sigma\in[\tfrac14,\kappa],
$$
so that
$$
\int_{\frac14\pm iT}^{\kappa\pm iT}F(s)X^s\frac{ds}{s}
\ll_\varepsilon X^{\frac14}T^{\frac54+\varepsilon}+XT^{-1+\varepsilon}.
$$
The pole at $s=\frac12$ contributes a residue of size
$O_\varepsilon(X^{\frac12+\varepsilon})$, and
from the pole at $s=1$ we get the main term, which turns out to be
$\frac12\frac{\zeta^*(2)\zeta^*(3)}{\pi^{-1}\zeta^*(4)}X$.
Hence, altogether we have
$$
	\sum_{\substack{D\leq X\\\sqrt{D}\notin\Z}} 
	L(1, \psi_D) r(D) 
	=
	\frac{\pi}{2}\frac{\zeta^*(2) \zeta^*(3)}{\zeta^*(4)} X
+O_\varepsilon\!\left(
X^{\frac14}T^{\frac74+\varepsilon}
+X^{\frac12+\varepsilon}
+\frac{X^{1+\varepsilon}}{T}\right)
$$
for any $\varepsilon>0$. Taking $T=X^{\frac3{11}}$ yields the desired
bound.

To prove the meromorphic continuation of $F(s)$ and the estimates
\eqref{e:F14}, we compute from \eqref{eqn:level1minus} that
\begin{multline}\label{e:F1s}
	F(s)
	=
	\sum_{j=1}^\infty
	(-1)^{\epsilon_j} 
	\frac{L^*(2s, \Sym^2 f_j)}{\zeta^*(4s)\Gamma_\R(2s)}
	+
	\frac{\zeta^*(2s) \zeta^*(2s-1) \zeta^*(2s+1)}{\zeta^*(4s)\Gamma_\R(2s)}
	\\
	-
	\frac{1}{4\pi} \int_{\R} 
	\frac{\zeta(2s)\zeta^*(2s-2ir) \zeta^*(2s+2ir)}{\zeta^*(4s)}
	\frac{\phi'}{\phi}\!\left(\frac{1}{2}+ir\right) dr
	-
	\frac{1}{4}\frac{\zeta^*(2s)^3}{\zeta^*(4s)\Gamma_\R(2s)}
	\\
	-
	\sum_{\ell=1}^\infty
	\Biggl[\sum_{m\mid\ell} \Lambda(m)(1-m^{-1})
	+\tfrac12\bigl(\psi(s+\tfrac12)-\psi(s)\bigr)\Biggr]
	\frac{r(\ell^2)-2}{\ell^{2s}}
	, 
\end{multline}
where 
$
	L^*(2s, \Sym^2 f_j) 
	= 
	\Gamma_\R(2s) \Gamma_\R(2s-2ir_j) \Gamma_\R(2s+2ir_j)
	\zeta(4s) 
	\sum_{n=1}^\infty \frac{\lambda_j(n^2)}{n^{2s}}
$.
We consider each term in turn.

First, by Stirling's formula, for $s=\frac14+it$ we have
$$
\frac{\Gamma_\R(2s-2ir)\Gamma_\R(2s+2ir)}{\Gamma_\R(4s)}
\ll(1+|t^2-r^2|)^{-\frac14}e^{-\pi\max(0,|r|-|t|)},
$$
uniformly for $r,t\in\R$. Moreover,
$1/\zeta(4s)\ll\log(1+|t|)$, and by the uniform convexity bound
\cite[Cor.~2]{harcos}, we have
$$
L(2s,\Sym^2{f_j})\ll_\varepsilon
\bigl((1+|t|)(1+|t^2-r_j^2|)\bigr)^{\frac14+\varepsilon}.
$$
From these estimates and the Weyl bound
$\#\{j:|r_j|\leq T\}\ll T^2$, we see that the sum over
$j$ in \eqref{e:F1s} is $O_\varepsilon\big((1+|t|)^{\frac94+\varepsilon}\big)$,
as claimed. Moreover, by Cauchy--Schwarz, we have
\begin{align*}
\left(\sum_{j=1}^\infty\int_{-T}^T\left|
\frac{L^*(\frac12+2it,\Sym^2f_j)}{\zeta^*(1+4it)\Gamma_\R(\frac12+2it)}
\right|dt\right)^2\ll&\log^2T
\sum_{j=1}^\infty\int_{-T}^T
(1+|t^2-r_j^2|)^{-\frac12}e^{-\pi\max(0,|r_j|-|t|)}\,dt\\
&\cdot
\sum_{j=1}^\infty\int_{-T}^T
|L(\tfrac12+2it,\Sym^2f_j)|^2e^{-\pi\max(0,|r_j|-|t|)}\,dt\\
&\hspace{-2cm}
\ll(T\log{T})^2
\sum_{j=1}^\infty\int_{-T}^T
|L(\tfrac12+2it,\Sym^2f_j)|^2e^{-\pi\max(0,|r_j|-|t|)}\,dt.
\end{align*}
By \cite[Thm.~2]{kss} and \cite[Cor.~C]{ramakrishnan-wang}, we have
$$
\int_{-T}^T|L(\tfrac12+2it,\Sym^2f_j)|^2\,dt
\ll_\varepsilon(1+r_j)^\varepsilon T^{\frac32}\log{T},
$$
and altogether this yields
$$
\sum_{j=1}^\infty\int_{-T}^T\left|
\frac{L^*(\frac12+2it,\Sym^2f_j)}{\zeta^*(1+4it)\Gamma_\R(\frac12+2it)}
\right|dt
\ll_\varepsilon T^{\frac{11}4+\varepsilon}.
$$

For all remaining terms we obtain a pointwise bound of at most
$O_\varepsilon\bigl((1+|t|)^{\frac54+\varepsilon}\bigr)$,
which suffices for both estimates in
\eqref{e:F14}. First, we have
$$
\frac{\zeta(2s) \zeta^*(2s-1) \zeta^*(2s+1)}{\zeta^*(4s)}
\ll_\varepsilon(1+|t|)^{\frac34+\varepsilon}.
$$
Next, by a similar analysis to the proof of Proposition~\ref{prop:tf1},
the second line of \eqref{e:F1s} can be written as
\begin{align*}
\frac12\zeta(2s)\frac{{\zeta^*}'}{\zeta^*}(2s+1)
&-\frac12\frac{\zeta^*(4s-1){\zeta^*}'(2s)}{\zeta^*(4s)\Gamma_\R(2s)}\\
&-\frac1{2\pi i}\int_{\Re(u)=\sigma}
\frac{\zeta(2s)\zeta^*(2s-2u) \zeta^*(2s+2u)}{\zeta^*(4s)}
\frac{{\zeta^*}'}{\zeta^*}(1+2u)\,du,
\end{align*}
for any $\sigma>\frac14$, and the integral is
$\zeta(2s)/\zeta^*(4s)$ times an analytic function for
$\Re(s)\in(\frac12-\sigma,\sigma)$.  The first two terms are analytic for
$\Re(s)\ge\frac14$ apart from a pole at $s=\frac12$, and by the
convexity bounds for $\zeta(\frac12+2it)$ and $\zeta'(\frac12+2it)$,
we see that they are $O_\varepsilon\big((1+|t|)^{\frac14+\varepsilon}\big)$
for $s=\frac14+it$. Taking $\sigma=\frac14+\varepsilon$ and writing
$s=\frac14+it$, $u=\frac14+\varepsilon+ir$, we find by a similar
analysis to the above that
$$
\frac{\zeta(2s)\zeta^*(2s-2u) \zeta^*(2s+2u)}{\zeta^*(4s)}
\frac{{\zeta^*}'}{\zeta^*}(1+2u)
\ll_\varepsilon(1+|t|)^{\frac14+\varepsilon}
(1+|t^2-r^2|)^\varepsilon
e^{-\pi\max(0,|r|-|t|)},
$$
so the integral is $O_\varepsilon\big((1+|t|)^{\frac54+\varepsilon}\big)$.

Turning to the third line of \eqref{e:F14}, note first that
$\sum_{m\mid\ell}\Lambda(m)=\log\ell$, and
$$
-\sum_{\ell=1}^\infty\frac{(r(\ell^2)-2)\log\ell}{\ell^{2s}}
=\frac12\frac{d}{ds}\left(
\frac{\zeta^2(2s)L(2s,\psi_{-4})}{\zeta(4s)}-2\zeta(2s)\right).
$$
Again using the convexity bound, this is
$O_\varepsilon\big((1+|t|)^{\frac34+\varepsilon}\big)$ for
$s=\frac14+it$. Similarly,
$$
-\frac12\bigl(\psi(s+\tfrac12)-\psi(s)\bigr)
\sum_{\ell=1}^\infty\frac{r(\ell^2)-2}{\ell^{2s}}
=-\frac12\bigl(\psi(s+\tfrac12)-\psi(s)\bigr)
\left(\frac{\zeta^2(2s)L(2s,\psi_{-4})}{\zeta(4s)}-2\zeta(2s)\right),
$$
and this is $O_\varepsilon\big((1+|t|)^{\frac14+\varepsilon}\big)$ for
$s=\frac14+it$. Finally, we have
$$
-2\sum_{\ell=1}^\infty\ell^{-2s}\sum_{m\mid\ell}\frac{\Lambda(m)}{m}
=2\zeta(2s)\frac{\zeta'}{\zeta}(2s+1),
$$
which is $O_\varepsilon\big((1+|t|)^{\frac14+\varepsilon}\big)$ for
$s=\frac14+it$, and
$$
\sum_{\ell=1}^\infty r(\ell^2)\ell^{-2s}\sum_{m\mid\ell}\frac{\Lambda(m)}{m}
=\frac{\zeta^2(2s)L(2s,\psi_{-4})}{\zeta(4s)}
\sum_p\frac{\log{p}}{p-1}
\left[1-\frac{E_p(p^{-2s-1})}{E_p(p^{-2s})}\right],
$$
where $E_p(T)=\frac{1+T}{(1-T)(1-\psi_{-4}(p)T)}$.
The sum over $p$ is analytic and bounded for
$\Re(s)\ge\frac14$, so the last line is
$O_\varepsilon\big((1+|t|)^{\frac34+\varepsilon}\big)$
for $s=\frac14+it$.
\qed

\bibliographystyle{alpha}
\bibliography{reference_Selberg}
\end{document}